\documentclass{amsart}[11pt]

\usepackage{amsmath}
\usepackage{amssymb}
\usepackage{amsthm}
\usepackage[mathscr]{eucal}
\usepackage{mathrsfs}
\usepackage{psfrag}
\usepackage{fullpage}
\usepackage{color}
\usepackage{eucal}
\usepackage[dvips]{graphicx}
\usepackage{amsfonts}%
\usepackage{amsaddr}
\usepackage{graphicx}
\usepackage{graphics}
\usepackage{enumerate}
\usepackage{dsfont}

\usepackage[margin=1in]{geometry}
\usepackage{float}
\graphicspath {{/}}

\newtheorem{lemma}{Lemma}
\newtheorem{theorem}{Theorem}
\newtheorem{definition}{Definition}
\newtheorem{condition}{Condition}

\newtheorem{remark}{Remark}

\newcommand{\xa}{X^{\varepsilon,\theta, H}}

\newcommand{\xia}{\xi^{\theta, H}}
\newcommand{\xian}{\xi^{\theta, H}_{(n)}}
\newcommand{\Xian}{\Xi^{\theta, H}_{(n)}}
\newcommand{\xazero}{X^{\varepsilon, \theta_0, H_0}}

\newcommand{\xiazero}{\xi^{\theta_0, H_0}}
\newcommand{\xianzero}{\xi^{\theta_0, H_0}_{(n)}}

\newcommand{\Xianzero}{\Xi^{\theta_0, H_0}_{(n)}}
\newcommand{\XianzerocurlyH}{\Xi^{\theta_0, \mathcal{H}}_{(n)}}
\newcommand{\XianthetadaggermathcalH}{\Xi^{\theta^\dagger, \mathcal{H}}_{(n)}}
\newcommand{\xe}{X^\varepsilon}
\newcommand{\ye}{Y^\eta}

\newcommand{\tnk}{{t^n_k}}
\newcommand{\tnkm}{{t^n_{k-1}}}
\newcommand{\tnkmm}{{t^n_{k-2}}}
\newcommand{\tnj}{{t^n_j}}
\newcommand{\nk}{{n,k}}
\newcommand{\Deltank}{\Delta_\nk}
\newcommand{\Deltatwonk}{\Delta^{(2)}_\nk}

\newcommand{\tone}{{t_1}}
\newcommand{\ttwo}{{t_2}}
\newcommand{\sone}{{s_1}}
\newcommand{\stwo}{{s_2}}

\newcommand{\var}{\operatorname{Var}}
\newcommand{\cov}{\operatorname{Cov}}

\newcommand{\norm}[1]{\left\lVert #1 \right\rVert}

\begin{document}
\bibliographystyle{plain}
	
	
\author{S. Bourguin, S. Gailus, and K. Spiliopoulos}
\address{Boston University, Department of Mathematics and Statistics\\ 111 Cummington Mall, Boston, MA 02215, USA}
\email[Solesne Bourguin]{bourguin@math.bu.edu}
\email[Siragan Gailus]{siragan@math.bu.edu}
\email[Konstantinos Spiliopoulos]{kspiliop@math.bu.edu}
	
\title{Discrete-time inference for slow-fast systems driven by fractional Brownian motion}
	
\thanks{S. Bourguin was supported in part by the Simons Foundation (Grant 635136), K. Spiliopoulos was supported in part by the National Science Foundation (DMS 1550918) and by the Simons Foundation (Grant 672441)}
	\date{\today}
\begin{abstract}
We study statistical inference for small-noise-perturbed multiscale dynamical systems where the slow motion is driven by fractional Brownian motion. We develop statistical estimators for both the Hurst index as well as a vector of unknown parameters in the model based on a single time series of observations from the slow process only. We prove that these estimators are both consistent and asymptotically normal as the amplitude of the perturbation and the time-scale separation parameter go to zero. Numerical simulations illustrate the theoretical results.
\end{abstract}
	

\maketitle
	

\section{Introduction}

In this work we consider statistical estimation for small-noise perturbations of multiscale dynamical systems. The main feature of the model is that the random perturbation of the slow motion of the system arises from a fractional Brownian motion (fBm), thereby making it possible to capture dynamical features that are out of the scope of the standard Brownian motion. More precisely, we consider $(\xe, \ye)_T = \{(\xe_t,
\ye_t)\}_{0\leq t\leq T}$ evolving in
$\mathcal{X}\times\mathcal{Y}:=\mathbb{R}^{m}\times\mathbb{R}^{d-m}$
according to the stochastic differential equation
\begin{equation}
\label{Eq:ModelSystem}
\begin{cases} d\xe_t =
c_{\theta}(\xe_t, \ye_t) dt + \sqrt\epsilon \sigma(\ye_t) dW^H_t \\ d\ye_t = \frac{1}{\eta} f(\ye_t) dt
+ \frac{1}{\sqrt\eta} \tau(\ye_t) dB_t \\
\xe_0=x_0\in\mathcal{X}, \hspace{1pc} \ye_0=y_0\in\mathcal{Y}.
\end{cases}
\end{equation}

Here, $W^H$ is a fractional Brownian motion (fBm) with Hurst index
$H\in(1/2,1)$ and $B$ is a standard Brownian motion (Bm) independent
of $W^H$. The term $dW^H$ is to be understood in the sense of pathwise
integration, although this pathwise integral coincides in our
framework with the analogous divergence
integral. $\varepsilon:=(\epsilon,\eta)\in\mathbb{R}^2_+$ is a pair of
small positive parameters. Note that $\eta$ is the time-scale
separation parameter while $\epsilon$ dictates the size of the noise.

In this work, we study estimation of the Hurst index $H\in(1/2,1)$ and the vector $\theta\in \Theta$, where $\Theta$ is an open, bounded, and convex subset of a Euclidean space, based on a time-series sampled from a realization of the slow process $\xe$ in the regime of $\varepsilon := (\epsilon,\eta) \to 0$.

It is widely understood that data coming from physical systems can exhibit multiple characteristic scales in time or space. Stochastic noise is often included to account for uncertainty or as an intrinsic feature of a given model. It is perhaps most common to model the noise with standard Brownian motion. In this work we consider instead the choice of fractional Brownian motion with a Hurst index $H$ that may not be known but that must rather be estimated from empirical observations. It is well known that fBm is a one-parameter extension of standard Bm, which is recovered when $H=1/2$. Multiscale SDE models like
(\ref{Eq:ModelSystem}) with $H=1/2$ (i.e., models in which $\xe$ is perturbed by a standard Bm) have been applied in a wide variety of fields, including chemistry, biology, neuroscience, meteorology, econometrics, and mathematical finance. The reader is encouraged to consult \cite{chauviere2010cell,jean2000derivatives,janke2008rugged,jirsa2014nature,majda2008applied,zhang2005tale,zwanzig1988diffusion}, to offer just a few examples in the applied literature.

A model like (\ref{Eq:ModelSystem}) can be interpreted as perturbation of an underlying deterministic dynamical
system, say $\dot{\bar{X}} = \bar{c}(\bar{X})$, by small noise and multiple scales, see for example \cite{FS, FWBook, jirsa2014nature, Spiliopoulos2012,zwanzig1988diffusion}. To be more specific, consider the situation where $\bar{c}(x)$ is given as  the
integral of a given function with respect to a given measure $\mu$. In such a situation, the dynamical system $\dot{\bar{X}} = \bar{c}(\bar{X})$ is in fact a small noise perturbation of a system of slow and fast motion and  the measure
$\mu$ is the invariant measure of the fast motion. See also \cite{guy2014parametric} for further motivation on models with small noise diffusion. In addition, models like  (\ref{Eq:ModelSystem}) also arise when one considers systems with mulitple
scale but one is interested in small time asymptotics, see for example \cite{feng2012small, spiliopoulos2014fluctuation}. 

Whereas the standard Bm case has been studied extensively, the mathematical theory of multiscale models with $H\neq 1/2$ is considerably less developed. In our recent work \cite{fluctuations_fbm_multiscale}, we obtained results on averaging, homogenization, and fluctuations corrections for models like (\ref{Eq:ModelSystem}) considered in the limit as $\varepsilon := (\epsilon,\eta) \to 0$; see also \cite{hl:averagingdynamics, PeiInahamaXu2020} for related averaging results.

In the standard Bm case $H=1/2$, estimation of $\theta$ has been well studied both in the absence of mutiple scales, when $\eta \equiv 1$, as well as the presence of multiple scales, when $\eta \to 0$. For the case of $H=1/2$ and $\eta \equiv 1$, we refer the interested reader to \cite{bishwal2008parameter,guy2014parametric, KutoyantsSmallNoise, KutoyantsStatisticalInference,rao1999statistical,SorensenUchida,Uchida} for estimation of the unknown vector $\theta$ based on continuously and discretely observed data. For the case $H=1/2$ and $\eta \to 0$, we refer the interested reader to \cite{azencott2013sub,azencott2010adaptive,gs:disctime,gs:statinf,krumscheid2013semiparametric,papavasiliou2009maximum,pavliotis2007parameter, spiliopoulos2013maximum} for estimation of $\theta$ in the presence of multiple scales based on continuously and discretely-observed data.

In the case of $H\neq 1/2$ and $\eta \equiv 1$, statistical estimation of the Hurst parameter $H$ has recently been studied, under various assumptions, in works such as \cite{Berzin2014,ChronopoulouViens2009,Coeurjolly2001,GairingEtAll2019,istaslang,kubiliusskorniakov,kubiliusskorniakovMelichov}.

The focus of this paper is different. We develop the theory of estimation of both $H$ and $\theta$ for multiscale models like (\ref{Eq:ModelSystem}) in the multidimensional case. Our main assumption is that we are given only a discrete-time sample $\{x_{t_{k}}\}_{k=1}^{n}$ from a single observation of the slow process $\xe$; we assume that no data are available from the fast process $\ye$. We develop strongly consistent and asymptotically normal statistical estimators for both $H$ and $\theta$ in the asymptotic regime in which $\varepsilon := (\epsilon,\eta) \to 0$. The limiting variance is precisely calculated in all cases. To the best of our knowledge, this is the first paper on discrete-time estimation of multidimensional multiscale models perturbed by fractional Brownian motion.


Note that our analysis is focused on the case where $H \in (1/2,1)$. This
comes from the fact that this is the only range of values for $H$ where
the interpretation of stochastic integrals as pathwise integrals or as
divergence integrals can coexist. We need the pathwise interpretation
for justifying the existence of a solution to our slow-fast system \eqref{Eq:ModelSystem}
and the divergence interpretation to exploit the tools of the
Malliavin stochastic calculus of variations applied to
(transformations) of the solution of \eqref{Eq:ModelSystem}. It is worth mentioning that
a possible approach for the case $H \in (1/4,1/2)$ could be the so-called
rough paths interpretation of the stochastic integral with respect to
fBm appearing in our system \eqref{Eq:ModelSystem}, but the interactions with Malliavin
calculus that are needed here do not exist. For $H \in (0,1/4]$, very little
is known in terms of solving \eqref{Eq:ModelSystem}.We also mention that the assumption
of independence of $B$ and $W^H$ comes from reasons that are technical in
nature. Indeed, some crucial estimates used in the proofs of the
theoretical results obtained in our recent work \cite{fluctuations_fbm_multiscale} necessitate that
the pathwise and divergence integrals with respect to the fBm $W^H$
coincide, which only happens when one assumes independence. As we make
use of some of these results here, we need the same assumptions. The
introduction of \cite{fluctuations_fbm_multiscale} contains a more detailed and technical discussion
of the reasons leading to the independence assumption, which we refer
the reader to.

The rest of the paper is organized as follows. In Section \ref{S:Assumptions}, we describe our assumptions and related preliminary results. In Section \ref{S:EstimationH}, we present and analyze two estimators for the Hurst index $H$, one that requires knowledge of the magnitude $\epsilon$ of the perturbing noise and one that does not require this knowledge. The trade-off is that the latter has higher limiting variance than the former. Neither estimator for $H$ requires knowledge of $\theta$. In Section \ref{S:TFE1}, we present and analyze a trajectory-fitting estimator (TFE) for the estimation of $\theta$, which does not require knowledge of $H$. We then study in Section \ref{S:DriftEstimationMLE} an estimator for $\theta$ which is based on the principle of maximum likelihood applied to the fluctuations results of \cite{fluctuations_fbm_multiscale}. The advantage of this alternative estimator over the TFE presented in section \ref{S:TFE1} is that it has smaller limiting theoretical variance, at least when the Hurst index is known. One clear disadvantage is that its implementation is computationally more challenging in that it involves inversion of matrices which are very large in typical cases. The estimator is nevertheless of theoretical interest, as well as practical interest in cases in which the computational challenges can be met. In Section \ref{S:NumericalExamples}, we present data from numerical simulations to illustrate the theoretical results. For the convenience of the reader we have included in Appendix \ref{B:Appendix} and \ref{A:Appendix} respectively a technical lemma and an introduction to fBm and necessary results from the Malliavin calculus.

\section{Conditions and preliminary results}\label{S:Assumptions}
In this section we introduce notation, present the conditions that we will assume throughout the paper, and recall necessary results of \cite{fluctuations_fbm_multiscale}. We shall suppress the dependence on $\theta$ for expository purposes.

We will denote by $A:B$ the Frobenius inner product $\Sigma_{i,j}[a_{i,j}\cdot b_{i,j}]$ of matrices $A=(a_{i,j})$ and $B=(b_{i,j})$.
We will use single bars $|\cdot|$ to denote the Frobenius (or Euclidean) norm of a matrix, and double bars $||\cdot||$ to denote the operator norm.

Condition \ref{c:regularity} below imposes conditions of growth and regularity on the drift and diffusion coefficients of the model.
\begin{condition}\label{c:regularity}\hspace{1pc}\newline
\hspace{1pc}\newline
\noindent Conditions on $c_\theta$:
\begin{enumerate}[-]
	\item $\exists$\hspace{0.5pc}$(K,q)\in\mathbb{R}^2_+$, $r\in[0,1)$; $\forall \theta \in \Theta$, $\left|c(x,y)\right|\leq K(1+|x|^r)(1+|y|^q)$
	\item $\exists$\hspace{0.5pc}$(K,q)\in\mathbb{R}^2_+$; $\forall \theta \in \Theta$, $\left|\nabla_x c(x,y)\right|+\left|\nabla_x\nabla_x c(x,y)\right|\leq K(1+|y|^q)$
	\item $c$, $\nabla_x c$, $\nabla_x \nabla_x c$, and $\nabla_y \nabla_y c$ are continuous in $(\theta, x,y)$
	\item $c$, $\nabla_x c$, and $\nabla_x \nabla_x c$ are H\"older continuous in $y$ uniformly in $(\theta, x)$
	\item $c$ has two locally-bounded derivatives in $\theta$ with at most polynomial growth in $x$ and $y$
\end{enumerate}

\noindent Conditions on $\sigma$:
\begin{enumerate}[-]
	\item $\exists$\hspace{0.5pc}$(K,q)\in\mathbb{R}^2_+$; $|\sigma(y)|\leq K(1+|y|^q)$
	\item $\sigma\sigma^T$ is uniformly nondegenerate
\end{enumerate}

\noindent Conditions on $f$ and $\tau$:
\begin{enumerate}[-]
	\item $f$ and $\tau\tau^T$ are twice differentiable, and, along with their first and second derivatives, are H\"older continuous
	\item $\tau\tau^T$ is uniformly bounded and uniformly nondegenerate.
\end{enumerate}

\end{condition}
Condition \ref{c:recurrencebasic} is a basic condition of recurrence type on the fast component, yielding ergodic behavior.
\begin{condition}\label{c:recurrencebasic}\hspace{1pc}\newline
\begin{align*}
\lim_{|y|\to\infty} y \cdot f(y) &= -\infty.
\end{align*}
\end{condition}

To derive most of our results we shall in fact assume a stronger recurrence condition.
\begin{condition}\label{c:recurrence}\hspace{1pc}\newline
For real constants $\alpha>0$, $\beta\geq2$, and $\gamma>0$, we shall write:
\begin{enumerate}[-]
	\item Condition \ref{c:recurrence}-$(\alpha,\beta)$: one has
	\begin{align*}
	y \cdot f(y) + \alpha |y|^\beta + \frac12 (\beta-2+d-m)\sup_{\tilde y\in\mathcal{Y}}|\tau(\tilde y)|^2 \leq 0
	\end{align*}
	for $|y|$ sufficiently large
	\item Condition \ref{c:recurrence}-$(\alpha,\beta,\gamma)$:
	Condition \ref{c:recurrence}-$(\alpha,\beta)$ holds and, moreover, one has $||\nabla_x c(x,y)|| \leq \gamma|y|^\beta$ for $|y|$ sufficiently large.
\end{enumerate}
\end{condition}

In Section \ref{S:NumericalExamples}  we present numerical examples for specific systems satisfying Conditions \ref{c:regularity}-\ref{c:recurrence}. The numerical results demonstrate that the theoretical results should also hold when the growth of $c_{\theta}$ is linear in $x$, i.e., when $r=1$ in Condition \ref{c:regularity} instead of the sub-linear growth assumption under which we are able to prove our results.

\begin{remark} Clearly, Condition \ref{c:recurrencebasic} is implied by Condition \ref{c:recurrence}-$(\alpha, \beta)$, which in turn is implied by the stronger condition
\begin{align*}
\lim_{|y|\to\infty} y \cdot f(y) + \alpha |y|^\beta &= -\infty.
\end{align*}
\end{remark}
One has the infinitesimal generator
\begin{equation*}
\mathcal{L}:=f\cdot\nabla_y+\frac12(\tau\tau^T):\nabla_y^2
\end{equation*}
for the rescaled fast dynamics. Conditions \ref{c:regularity} and \ref{c:recurrencebasic} are enough to guarantee that one has on $\mathcal{Y}$ a unique invariant measure $\mu$ corresponding to the operator $\mathcal{L}$, as discussed for example in \cite{ReyBellet2006}.

\begin{remark}
Therefore, in particular, the process $\ye$, obtained as the solution of an SDE that does not depend on $\xe$, does not explode and is well defined for all times. Meanwhile, Condition \ref{c:regularity} guarantees that the drift coefficient of $\xe$ is Lipschitz continuous in the variable $x$ locally in the variable $y$. Thus one sees that our assumptions are sufficient to guarantee that $\xe$ is well defined on $[0, T]$ (compare the situation with, e.g., \cite[Sections 2 and 4]{pardoux2001poisson}). For general results on existence and uniqueness of solutions of equations with standard and fractional Brownian motions, see for instance \cite{guerranualart, kubiliusfractionalbrownianmotion, kubiliuspsemimartingale, Mishurashevchenko}.
\end{remark}

Finally, let us recall the main convergence result of
\cite{fluctuations_fbm_multiscale}, which we shall use frequently
in this work. By \cite[Theorem 3]{pardoux2003poisson}, the equations
\begin{align}
\mathcal{L}\Phi(x,y) &= -\left( c(x,y)-\bar{c}(x) \right) \nonumber\\
 \int_{\mathcal{Y}}\Phi(x,y)&\mu(dy)=0,\label{poissonequation0}
\end{align}
 where $\bar c$ is the averaged function
\begin{align*}
\bar c(x) &:= \int_{\mathcal{Y}}c(x,y)\mu(dy),
\end{align*}
admit a unique solution $\Phi$ in the class of functions that grow at most polynomially in $|y|$ as $y\to\infty$.

\begin{theorem}[Theorems 1 and 2 in \cite{fluctuations_fbm_multiscale}]\label{T:LimitBehavior}
Assume Conditions \ref{c:regularity} and \ref{c:recurrence}-$(\alpha,\beta,\gamma)$, where $\alpha \geq 0$, $\beta \geq 2$, $\gamma \geq 0$, and $T \beta \gamma \sup_{y \in \mathcal{Y}} ||\tau(y)||^2 < 2 \alpha$. For any $0 < p < \frac{2 \alpha}{T \beta \gamma \sup_{y \in \mathcal{Y}} \|\tau(y)\|^2}$, there is a constant $\tilde K$ such that for $\varepsilon:=(\epsilon,\eta)$ sufficiently small,
\begin{align*}
E\sup_{0\leq t\leq T}\left|\xe_t - \bar X_t\right|^p&\leq \tilde K \left( \sqrt\epsilon^p + \sqrt\eta^p \right),
\end{align*}
where $\bar X$ is the (deterministic) solution of the integral equation
\begin{align*}
\bar X_t &= x_0 + \int^t_0 \bar c(\bar X_s)ds.
\end{align*}

Now suppose in addition that $\eta = \eta(\epsilon)$ and that $\lim_{\epsilon\to0}\frac{\sqrt{\eta}}{\sqrt{\epsilon}} =: \lambda \in [0,\infty)$. Let $\Phi$ be as in \eqref{poissonequation0} and set $\Sigma_\Phi := (\overline{(\nabla_y\Phi\tau)(\nabla_y\Phi\tau)^T})^{1/2}$. The family of processes $\left\{\xi^\varepsilon:=\frac{1}{\sqrt{\epsilon}}(\xe-\bar{X})\right\}_\epsilon$ converges in distribution on the space $C([0,T]; \mathcal{X})$ (endowed, as usual, with the topology of uniform convergence) as $\epsilon\to0$ to the law of the solution $\xi$ of the mixed SDE
\begin{align*}
\begin{cases}
\displaystyle \xi_t = \int^t_0 (\nabla_x \bar c)(\bar X_s) \cdot \xi_s ds + \lambda \int^t_0 \Sigma_\Phi(\bar X_s) d\tilde{B}_s +  \bar \sigma  \tilde{W}^H_t\\
\displaystyle \xi_0=0
\end{cases},
\end{align*}
where $\bar \sigma := \int_{\mathcal{Y}}\sigma(y)\mu(dy)$, $\tilde{W}^H$ is a fractional Brownian motion with Hurst
index $H$, and $\tilde{B}$ is a standard Brownian
motion independent of $\tilde{W}^H$.
\end{theorem}
\section{Estimation of the Hurst parameter H}\label{S:EstimationH}
This section is devoted to introducing and studying two estimators of
the Hurst parameter of the fractional Brownian motion appearing in our
slow-fast system \eqref{Eq:ModelSystem}. The first estimator (Theorem \ref{T:ConvergenceForHatH1}) that we
consider requires knowledge of the magnitude of the noise
$\epsilon>0$ (which is not necessarily always possible depending on
the framework considered in practice), whereas the second one (Theorem \ref{T:ConvergenceForHatH2}) does not require knowledge of the
magnitude of the noise $\epsilon>0$. Before stating the main results
of this section, we introduce some notation as well as a few preliminary lemmas to be used in the
proofs of the two main results dealing with the asymptotic properties
of the two estimators.

~
\newline
For each $n\in\mathbb{N}$ and $k\in\{0, 1, \cdots, n\}$, let $\tnk:=Tk/n$. We denote first- and second-order filtered observations of a stochastic process $\{z_t\}_{0 \leq t \leq T}$ by
\begin{align*}
\Deltank z &:= z_\tnk-z_\tnkm,\\
\Deltatwonk z &:= z_\tnk-2z_\tnkm+ z_\tnkmm.
\end{align*}
We first wish to obtain control, as $\varepsilon\to0$ and
$n\to\infty$, over the discrepancy between the behaviour of
$\frac{1}{\epsilon}(\Deltatwonk\xe)^2$ and that of
${\bar\sigma}^2(\Deltatwonk W^H)^2$. To this end, we have the
following three lemmas.
\begin{lemma}\label{L:FilteredDifferences}
Let $\xe$ and $\ye$ be as in (\ref{Eq:ModelSystem}). Assume Conditions \ref{c:regularity} and \ref{c:recurrencebasic}. For any $ 1 \leq p < \infty $ and $ \zeta > 0 $,
\begin{align}
\Deltatwonk \xe &-\sqrt{\epsilon} \Deltatwonk\left[ \int^\cdot_0 \sigma(\ye_s) dW^H_s \right]= \frac{1}{n}o( \eta^{ - \zeta } )\nonumber
\end{align}
in $L^p(\Omega)$ as $\eta\to0$, uniformly in $n\in\mathbb{N}$, uniformly in $k\in\{2, \cdots, n\}$, and uniformly in $\epsilon\in(0,1]$.
\end{lemma}

\begin{proof}[Proof of Lemma \ref{L:FilteredDifferences}]
We will show that the estimate is valid for the first-order filters, i.e., that one has
\begin{align}
\Deltank \xe &-\sqrt{\epsilon} \Deltank\left[ \int^\cdot_0 \sigma(\ye_s) dW^H_s \right]= \frac{1}{n}o( \eta^{ - \zeta } ), \nonumber
\end{align}
uniformly in $k\in\{1, \cdots, n\}$; it is clear that the same estimate may then be used for the second-order filters.
By assumption, there are positive constants $K, q, r$ for which $ |c(x,y)| \leq K ( 1 + |x|^r + |y|^q ) $, so that
\begin{align*}
E \left| \int^\tnk_\tnkm c(\xe_s, \ye_s) ds \right| & \leq \frac{KT}{n} \left( 1 + E \sup_{ \tnkm \leq t \leq \tnk } |\xe_t|^r + E \sup_{ \tnkm \leq t \leq \tnk } |\ye_t|^q \right), \\
\end{align*}
whence by \cite[Lemma 2]{fluctuations_fbm_multiscale} and Lemma \ref{ymaximal}, it follows that for any $ \zeta > 0 $, one has
\begin{align*}
n \cdot \Deltank\left[\int^\cdot_0c(\xe_s, \ye_s) ds\right] & = o( \eta^{ - \zeta } )
\end{align*}
in $L^1(\Omega)$ as $\eta\to0$, uniformly in $n\in\mathbb{N}$,
uniformly in $k\in\{1, \cdots, n\}$, and uniformly in
$\epsilon\in(0,1]$. By Jensen's inequality this is easily extended to
a statement in $L^p(\Omega)$. As mentioned earlier, the statement is also valid with second-order filters $\Deltatwonk$ in place of first-order filters $\Deltank$.
\end{proof}

\begin{lemma}\label{L:FilteredDifferencesDiffusion}
Let $\xe$ and $\ye$ be as in (\ref{Eq:ModelSystem}). Assume, for some $\alpha > 0$ and $\beta \geq 2$, Conditions \ref{c:regularity} and \ref{c:recurrence}-$(\alpha, \beta)$. For any $ 2 \leq p < \infty $, $ \kappa > 0 $, and $ 0 < \zeta < (\kappa/2) \wedge (1/p) $,
\begin{align}
\Deltatwonk \left[ \int^\cdot_0 \sigma(\ye_s) dW^H_s \right] &- \Deltatwonk\left[ \bar\sigma W^H \right] = \frac{1}{n^{H-\kappa}}o( \eta^{ \zeta } )\nonumber
\end{align}
in $L^p(\Omega)$ as $\eta\to0$, uniformly in $n\in\mathbb{N}$, uniformly in $k\in\{2, \cdots, n\}$, and uniformly in $\epsilon\in(0,1]$.
\end{lemma}

\begin{proof}[Proof of Lemma \ref{L:FilteredDifferencesDiffusion}]
That the estimate is valid for the first-order filters, i.e., that one has
\begin{align}
\Deltank \left[ \int^\cdot_0 \sigma(\ye_s) dW^H_s \right] &- \Deltank\left[ \bar\sigma W^H \right] = \frac{1}{n^{H-\kappa}}o( \eta^{ \zeta } ) , \nonumber
\end{align}
uniformly in $k\in\{1, \cdots, n\}$, follows almost directly from \cite[Lemma 4.17]{hl:averagingdynamics}, the difference being that in our setting we do not assume that $\sigma$ is uniformly bounded nor that $\ye$ begins at time $t=0$ in stationarity. Nevertheless, as explained in detail in the proof of \cite[Lemma 7]{fluctuations_fbm_multiscale}, the arguments may be carried over. Having established the estimate for the first-order filters $\Deltank$, it is easy to see that it is also valid with second-order filters $\Deltatwonk$ in their stead.
\end{proof}

\begin{lemma}\label{L:FilteredDifferencesSquare}
Assume Condition \ref{c:regularity} and Condition \ref{c:recurrencebasic} and, if $\sigma$ is nonconstant, assume that for some $\alpha > 0$ and $\beta \geq 2$ we also have Condition \ref{c:recurrence}-$(\alpha, \beta)$. Suppose that $n \to \infty$ and $\varepsilon:=(\epsilon,\eta) \to 0$ in such a way that
\begin{enumerate}
	\item there is a $\rho_1 > 0$ such that $\eta^{-1} = O((\epsilon n^{2-2H})^{\rho_1})$;
	\item if $\sigma$ is nonconstant, there is a $\rho_2 > 0$ such that $\eta = O(n^{-2-\rho_2})$.
\end{enumerate}
For any $ 1 \leq \tilde p < \infty $,
\begin{align}
\sup_{2 \leq k \leq n} \left| \left| \frac{1}{\sqrt\epsilon} \Deltatwonk\xe \right|^2-\left|{\bar\sigma} \Deltatwonk W^H \right|^2 \right| &= o(n^{-2H}) \nonumber
\end{align}
in $L^{\tilde p}(\Omega)$. Furthermore, if one has $\eta = O(\epsilon)$, the above can be strengthened to that there is $\rho_3 > 0$ such that
\begin{align}
\sup_{2 \leq k \leq n} \left| \left| \frac{1}{\sqrt\epsilon} \Deltatwonk\xe \right|^2-\left|{\bar\sigma} \Deltatwonk W^H \right|^2 \right| &= O(n^{-2H-\rho_3}) \nonumber
\end{align}
in $L^{\tilde p}(\Omega)$.
\end{lemma}

\begin{proof}[Proof of Lemma \ref{L:FilteredDifferencesSquare}]
We have for each $n \geq 2$ and $k \in \{2, \cdots, n\}$,
\begin{align}
\left|\frac{1}{\sqrt\epsilon} \Deltatwonk\xe \right|^2-\left|{\bar\sigma} \Deltatwonk W^H \right|^2 &= \langle A_{n,k}, B_{n,k} \rangle,\label{productform}
\end{align}
where
\begin{align*}
A_{n,k} &:= \Deltatwonk \left[ \frac{1}{\sqrt{\epsilon}}\xe - \bar\sigma W^H \right], \hspace{1pc} B_{n,k} := A_{n,k} + 2 \Deltatwonk \left[ \bar\sigma W^H \right],
\end{align*}
and where the angle brackets denote the usual inner product in $\mathbb{R}^m$.

Notice that $ 2 \Deltatwonk \left[ \bar\sigma W^H \right] $ is
$O(n^{-H})$ in $L^{2 \tilde p}(\Omega)$. By H\"older's inequality and
the triangle inequality it therefore suffices to show that $A_{n,k}$
is $o(n^{-H})$ in $L^{2 \tilde p}(\Omega)$. Writing
\begin{align*}
A_{n,k} = \Deltatwonk \left[ \frac{1}{\sqrt{\epsilon}}\xe - \int^\cdot_0 \sigma(\ye_s) dW^H_s \right] + \Deltatwonk \left[ \int^\cdot_0 \sigma(\ye_s) dW^H_s - \bar\sigma W^H \right],
\end{align*}
it suffices by the triangle inequality to check that each summand is
$o(n^{-H})$ in $L^{2 \tilde p}(\Omega)$. For the first summand this is
just Lemma \ref{L:FilteredDifferences} with $ p = 2 \tilde p $ and $
\zeta = 1 / \rho_1 $. If $\sigma$ is constant then for each $y \in \mathcal{Y}$, $\sigma(y) = \bar\sigma$ and one sees that the second summand is exactly $0$. Otherwise, the desired estimate for the second summand follows from Lemma \ref{L:FilteredDifferencesDiffusion} with $ p = 2 \tilde p $, $ \kappa = \tilde p^{-1} $, and $ \zeta = ( \tilde p ( 2 + \rho_2 ) )^{-1}$.
This concludes the proof of the first part of the claim.

The strenghthened estimate may be obtained by observing, firstly, that choosing different conjugate exponents in H\"older's inequality allows one to increase the value of $ \kappa $ in the appeal to Lemma \ref{L:FilteredDifferencesDiffusion}, and secondly, that the additional assumption $\eta = O(\epsilon)$ allows us to reduce the power of $n$ in the assumption $\eta^{-1} = O((\epsilon n^{2-2H})^{\rho_1})$, perhaps with a different choice of $\rho_1$. We omit the details because we do not use the strengthened estimate.

\end{proof}
We are now ready to present the results on the estimation of the Hurst parameter $H$. Theorem \ref{T:ConvergenceForHatH1} presents the asymptotic behavior of an estimator that requires knowledge of the magnitude of the noise $\epsilon>0$. On the other hand, Theorem \ref{T:ConvergenceForHatH2} presents the asymptotic behavior of an estimator that does not require knowledge of the magnitude of the noise $\epsilon>0$.

\begin{theorem}\label{T:ConvergenceForHatH1}
Assume Condition \ref{c:regularity} and Condition \ref{c:recurrencebasic} and, if $\sigma$ is nonconstant, assume that for some $\alpha > 0$ and $\beta \geq 2$ we also have Condition \ref{c:recurrence}-$(\alpha, \beta)$. Suppose that $n \to \infty$ and $\varepsilon:=(\epsilon, \eta) \to 0$ in such a way that
\begin{enumerate}
	\item there is a $\rho_1 > 0$ such that $\eta^{-1} = O((\epsilon n^{2-2H})^{\rho_1})$;
	\item if $\sigma$ is nonconstant, there is a $\rho_2 > 0$ such that $\eta = O(n^{-2-\rho_2})$.
\end{enumerate}
When $n\in\mathbb{N}$ satisfies $n>T$, the function $\phi_{n,T}:[0,1]\to[0,\infty)$ obtained by mapping $x\mapsto\left(\frac{T}{n}\right)^{2x}\left(4-2^{2x}\right)$ is strictly decreasing. Let us write $\phi_{n,T}^{-1}:[0,\infty)\to[0,1]$ for the left inverse of $\phi_{n,T}$ that restricts to an inverse on the actual image $\phi_{n,T}([0,1]) = [0,3]$ and is uniformly zero on $[3,\infty)$. For $n>T$, given a random sample $\{\xe_\tnk\}_{k=0}^{n}$, define the estimate
\begin{align}
\hat{H}^\epsilon_1(\{\xe_\tnk\}_{k=0}^{n})
&:=\phi_{n,T}^{-1}\left(\frac{1}{n\epsilon|\bar\sigma|^2}\sum_{k=2}^{n}\left|\Deltatwonk\xe\right|^{2}\right).\label{Eq:HatHEpsilonOneDefinition}
\end{align}
We have that $\hat{H}^\epsilon_1(\{\xe_\tnk\}_{k=0}^n) \rightarrow H$ in
                                          probability as
                                          $n\rightarrow\infty$ and $\varepsilon :=
                                          (\epsilon, \eta) \to 0$ and
\begin{equation*}
2\sqrt{n}\ln\left(\frac{n}{T}\right)\left(\hat{H}^\epsilon_1(\{\xe_\tnk\}_{k=0}^n)-
  H\right) \rightarrow \mathcal{N}(0,\varsigma^{2}_{\star}(H))
\end{equation*}
in distribution as $n\rightarrow\infty$ and $\varepsilon:=(\epsilon,
\eta) \to 0$, where the variance is given by
\begin{align*}
\varsigma_{\star}^2(H) :=
\varsigma^2_1(H)\left(\frac{1}{|\bar\sigma|^4} \sum_{i,k=1}^m\sum_{j,q=1}^{\tilde{m}}  \bar\sigma_{i, j}
\bar\sigma_{i, q} \bar\sigma_{k, j} \bar\sigma_{k,q}\right)
,
\end{align*}
and $m$ is the dimension of the slow process $X$, $\tilde m$ is the dimension of the noise $W^H$, while
\[
\varsigma^{2}_{1}(H):=2\sum_{j\in\mathbb{Z}}\rho^{2}(j;H), \textrm{ with }\rho(j;H):=\frac{-|j-2|^{2H}+4|j-1|^{2H}-6|j|^{2H}+4|j+1|^{2H}-|j+2|^{2H}}{2(4-2^{2H})}.
\]
\end{theorem}
\begin{remark}
Note that for any value of $H$, $\rho(j;H)$ is symmetric in $j$ and $\rho^2(0;H) = 1$, and so one may also write $\varsigma^{2}_{1}(H)=2\left(1+2\sum_{j=1}^{\infty}\rho^{2}(j;H)\right)$ as is sometimes done elsewhere in the literature, e.g. \cite{kubiliusskorniakovMelichov}.
\end{remark}
\begin{remark}\label{R:phi}
Note that $\phi^{-1}_{n,T}$ is a one-sided inverse only. Namely, $\phi^{-1}_{n,T}\circ\phi_{n,T}:[0,1]\to[0,1]$ is the identity map while $\phi_{n,T}\circ\phi^{-1}_{n,T}:[0,\infty)\to[0,\infty)$ maps $x\mapsto\min\{x,3\}$. The reason for this choice of domain for $\phi^{-1}_{n,T}$ is of course to ensure that the estimates are always defined. We will see that with probability approaching $1$ in the limit, one in fact lands in the invertible range.
\end{remark}
The idea is to use the approximation presented in Lemma \ref{L:FilteredDifferencesSquare}. Therefore, let us proceed as follows: we will first verify in Lemma \ref{L:SpecialConvergence} a basic convergence statement, then deduce in Lemma \ref{L:SpecialConvergenceForHatH1} consistency and asymptotic normality for the ideal case in which the data is sampled not from $\xe$ but rather from $\sqrt\epsilon\bar\sigma W^H$, and finally combine this with Lemma \ref{L:FilteredDifferencesSquare} to obtain Theorem \ref{T:ConvergenceForHatH1}.

In accordance with the above plan, let us start with the basic convergence statement. Note that there is no $\varepsilon$ here and the asymptotic regime of interest is simply $n\to\infty$.

\begin{lemma}\label{L:SpecialConvergence}
With notation as in the statement of Theorem \ref{T:ConvergenceForHatH1}, we have that
\begin{equation*}
\frac{n^{2H-1}}{T^{2H}(4-2^{2H})}\frac{1}{|\bar\sigma|^2}\sum^n_{k=2}\left|\bar\sigma
  \Deltatwonk W^H \right|^{2} \rightarrow 1 \textrm{ in probability
                   as } n\to\infty
\end{equation*}
and
\begin{equation*}
\sqrt{n}\left(\frac{n^{2H-1}}{T^{2H}(4-2^{2H})}\frac{1}{|\bar\sigma|^2}\sum^n_{k=2}\left|\bar\sigma \Deltatwonk W^H \right|^{2}-1\right)\rightarrow \mathcal{N}(0,\varsigma^{2}_{\star}(H)) \textrm{ in distribution as } n\to\infty.
\end{equation*}
\end{lemma}

\begin{proof}[Proof of Lemma \ref{L:SpecialConvergence}]
Recall that $m$ is the dimension of the Euclidean space in which the
slow process evolves. Let us write $\tilde m$ for the dimension of
$W^H$ and denote by $\{W^{H,j}\}^{\tilde m}_{j=1}$ the independent
components of $W^H$. For each pair $1\leq j, q \leq \tilde m$ define
\begin{equation*}
V^{j,q}_{n,T} := \frac{n^{2H-1}}{T^{2H}(4-2^{2H})}\sum^n_{k=2}\Deltatwonk W^{H,j}\Deltatwonk W^{H,q}.
\end{equation*}
With this notation we can write
\begin{equation}
\sqrt n \left(\frac{n^{2H-1}}{T^{2H}(4-2^{2H})}\frac{1}{|\bar\sigma|^2}\sum^n_{k=2}\left|\bar\sigma \Deltatwonk W^H \right|^{2} - 1\right) \label{Eq:EquationOne}
= A_{n,T} + B_{n,T},
\end{equation}
where
\begin{equation*}
A_{n,T} := \sqrt n
\left(\frac{1}{|\bar\sigma|^2}\sum^{m}_{i=1}\sum^{\tilde
    m}_{j=1}\bar\sigma^{2}_{i,j}V^{j,j}_{n,T} -1 \right) \quad\mbox{and}\quad
B_{n,T} := \sqrt n \left(\frac{1}{|\bar\sigma|^2}\sum^{m}_{i=1}\sum_{j\neq q}\bar\sigma_{i,j}\bar\sigma_{i,q}V^{j,q}_{n,T}\right).
\end{equation*}
Each of these belongs, at least asymptotically, to the second chaos of
the isonormal Gaussian process associated with $W^H$, and so it will
be enough to understand their limits separately. Let us start with
$A_{n,T}$. By known results in dimension one (see for example the references \cite{Coeurjolly2001, istaslang} and others in the introduction), for each $j$, we have that as $n\to\infty$,
\begin{equation}
  V_{n,T}^{j,j} \xrightarrow {a.s.} 1,\label{Eq:VAlmostSureLimit}
\end{equation}
and
\begin{equation}
\sqrt{n}\left( V_{n,T}^{j,j} -1 \right) \xrightarrow {\mathscr{D}} \mathcal{N}\left( 0,\varsigma_1^2(H) \right).\label{Eq:VDistributionLimit}
\end{equation}
Independence of $\{W^{H,j}\}^{\tilde m}_{j=1}$ implies independence of $\{V^{j,j}_{n,T}\}^{\tilde m}_{j=1}$, so that
\begin{align*}
\sqrt{n}\begin{pmatrix}
V_{n,T}^{1,1} -1 \\
\vdots \\
V_{n,T}^{\tilde m,\tilde m} -1
\end{pmatrix} \xrightarrow{\mathscr{D}} \mathcal{N} \left(\begin{pmatrix}
0 \\
\vdots \\
0
\end{pmatrix},   \begin{pmatrix}
   \varsigma_1^2(H)  & & \\
    & \ddots & \\
    & & \varsigma_1^2(H)
  \end{pmatrix}  \right).
\end{align*}
By the Cram\'er-Wold theorem (see \cite[Corollary 4.5]{kallenbergbook}), the known convergence (\ref{Eq:VDistributionLimit}), and the algebraic identity
\begin{align*}
A_{n,T} := \sqrt{n}\left(
  \frac{1}{|\bar\sigma|^2}\sum_{i=1}^{m}\sum_{j=1}^{\tilde m}\bar{\sigma}^{2}_{i,j}V_{n,T}^{j,j}
  -1\right) =
  \frac{1}{|\bar\sigma|^2}\sum_{j=1}^{\tilde m} \left(\sum_{i=1}^{m}
    \bar{\sigma}^{2}_{i,j}  \right)\sqrt{n} \left(V_{n,T}^{j,j} -1  \right),
\end{align*}
one then obtains
\begin{align}
A_{n,T} \xrightarrow {\mathscr{D}} \mathcal{N}\left(
  0,\frac{\varsigma_1^2(H)}{|\bar\sigma|^4}\sum_{i,k=1}^m \sum_{j=1}^{\tilde m} \bar{\sigma}^{2}_{i,j}\bar{\sigma}^{2}_{k,j}  \right). \label{Eq:EquationTwo}
\end{align}
Let us now show that $B_{n,T}$ also converges to a centered Gaussian
distribution as $n \to \infty$. It is straightforward to check that for each $1 \leq j \leq \tilde m$ and $2 \leq k \leq n$, one has in distribution
\begin{align*}
\Deltatwonk W^{H,j} &= I_1^{H,j}(h_k^n),
\end{align*}
where $I_1^{H,j}$ denotes the Wiener integral of order one with
respect to the fractional Brownian motion $W^{H,j}$ (see Appendix
\ref{SS:MultipleWienerIntergation} for an introduction and statement of key properties), and the function $h_k^n$ is given by $h_k^n =
\chi_{[t_{k-1}^n, t_k^n]} - \chi_{[t_{k-2}^n,
  t_{k-1}^n]}$, with $\norm{h_k^n}_{\frak{H}} =
\frac{T^H\sqrt{4-2^{2H}}}{n^H}$. Now, for any $1 \leq j \neq q \leq
\tilde m$ and $2 \leq k \leq n$, we have that $\left(\Deltatwonk
  W^{H,j},\Deltatwonk W^{H,q}  \right)$ is equal in distribution to
$\left( I_1(h_k^n),I_1(g_k^n) \right)$, where $I_1$ denotes a Wiener
integral of order one with respect to a generic fractional Brownian
motion with the same Hurst parameter $H$ as $W^H$, and where $g_k^n$ is a
function such that $\left\langle h_k^n,g_k^n \right\rangle_{\frak{H}}
= 0$ and $\norm{g_k^n}_{\frak{H}} = \norm{h_k^n}_{\frak{H}}$. This is explained in the next remark.
\begin{remark}
As the above equalities are only stated in distribution, we can
replace the Wiener integrals of the same function $h_k^n$ with respect
to the two independent fractional Brownian motions $W^{H,j}$ and
$W^{H,q}$ by Wiener integrals with respect to the same (generic) fractional
Brownian motion, but of orthogonal functions with the same norms, as
here orthogonality is a characterization of independence (see for
instance \cite[Proposition 1]{ustuzakai}). Hence, the
vectors $\left(I_1^{H,j}(h_k^n), I_1^{H,q}(h_k^n) \right)$ and $\left( I_1(h_k^n),I_1(g_k^n) \right)$ are equal in distribution.
\end{remark}
Based on this observation, we have that
$\sqrt{n}V_{n,T}^{j,q}$ has the same law as
\begin{equation*}
\frac{n^{2H-1}}{T^{2H}(4-2^{2H})}\sum^n_{k=2}I_1(h_k^n)I_1(g_k^n) =
\frac{n^{2H-1}}{T^{2H}(4-2^{2H})}\sum^n_{k=2}I_2(h_k^n
\widetilde{\otimes}g_k^n) = I_2 \left( \frac{n^{2H -
      \frac{1}{2}}}{T^{2H}(4-2^{2H})}\sum_{k=2}^n h_k^n
\widetilde{\otimes}g_k^n  \right),
\end{equation*}
where the first equality comes from applying the product rule for
Wiener integrals given in \eqref{multwienerintproductrule} together with the orthogonality in $\frak{H}$ of $h_k^n$ and $g_k^n$. Define
\begin{equation*}
\xi_{n,T} = \frac{n^{2H -
      \frac{1}{2}}}{T^{2H}(4-2^{2H})}\sum_{k=2}^n h_k^n
\widetilde{\otimes}g_k^n.
\end{equation*}
It is then straightforward to check that
\begin{equation*}
\var{\left(\sqrt{n}V_{n,T}^{j,q}  \right)} = \var{\left(I_2(\xi_{n,T})
  \right)} = 2 \norm{\xi_{n,T}}_{\frak{H}^{\otimes 2}}^2 \to \varsigma_1^2(H)
\end{equation*}
as $n \to \infty$. We will show that $\sqrt{n}V_{n,T}^{j,q} $ converges to a centered
Gaussian distribution with variance given by
$\varsigma_1^2(H)$. According to \cite[Theorem 5.2.7]{noupecbook}, all we need to prove is that the
contraction norm
$\norm{\xi_{n,T} \otimes_1 \xi_{n,T}}_{\frak{H}^{\otimes 2}} \to 0$
as $n \to \infty$ (see Appendix
\ref{SS:MultipleWienerIntergation} for a definition of contractions). We have
\begin{align*}
\xi_{n,T} \otimes_1 \xi_{n,T} &= \frac{n^{4H -
   1}}{T^{2H}(4-2^{2H})^2}\sum_{k_1,k_2=2}^n \left( h_{k_1}^n
\widetilde{\otimes}g_{k_1}^n \right) \otimes_1 \left( h_{k_2}^n
                                \widetilde{\otimes}g_{k_2}^n \right)\\
  &= \frac{n^{4H -
   1}}{T^{2H}(4-2^{2H})^2}\sum_{k=2}^n \left( h_{k}^n
\widetilde{\otimes}g_{k}^n \right) \otimes_1 \left( h_{k}^n
                                \widetilde{\otimes}g_{k}^n \right) + \frac{2n^{4H -
   1}}{T^{2H}(4-2^{2H})^2}\sum_{k=2}^n \left( h_{k}^n
\widetilde{\otimes}g_{k}^n \right) \otimes_1 \left( h_{k+1}^n
                                \widetilde{\otimes}g_{k+1}^n \right),
\end{align*}
where the reduction to only one sum comes from the fact that if $k_1$
and $k_2$ are strictly more than one appart, the resulting functions
are orthogonal. Repeating this argument when taking the norm in
$\frak{H}^{\otimes 2}$ of $\xi_{n,T} \otimes_1 \xi_{n,T}$ yields, with $C >0$ denoting a generic constant,
\begin{align*}
\norm{\xi_{n,T} \otimes_1 \xi_{n,T}}_{\frak{H}^{\otimes 2}}^2 \leq
  \frac{C n^{8H-1}}{T^{8H}(4-2^{2H})^4} \norm{h_k^n}_{\frak{H}}^8 =
  \frac{C}{n} \to 0
\end{align*}
as $n \to \infty$. Now, note that for any $j_1,j_2,q_1,q_2$ such that
$j_1 \neq j_2$, $j_1 < q_1$ and $j_2 < q_2$, it holds by the independence of $W^{H,j}$ and $W^{H,q}$ for $j \neq q$ that
\begin{equation*}
E \left( n V_{n,T}^{j_1,q_1} V_{n,T}^{j_2,q_2} \right) =0,
\end{equation*}
so that
all of the summands in $B_{n,T}$ are uncorrelated, and hence
asymptotically independent. We can therefore conclude that
\begin{equation}
  \label{Eq:EquationThree}
B_{n,T} \xrightarrow {\mathscr{D}} \mathcal{N} \left( 0, \frac{2\varsigma^2_1(H)}{|\bar\sigma|^4}  \sum_{j<q} \left(\sum^m_{i=1}\bar\sigma_{i,j}\bar\sigma_{i,q}\right)^2  \right).
\end{equation}
Taking together (\ref{Eq:EquationOne}), (\ref{Eq:EquationTwo}), and (\ref{Eq:EquationThree}), we conclude that
\begin{align*}
\sqrt n \left(\frac{n^{2H-1}}{T^{2H}(4-2^{2H})|\bar\sigma|^2}\sum^n_{k=2}\left|\bar\sigma \Deltatwonk W^H \right|^{2} - 1\right)
\end{align*}
converges to a centered Gaussian distribution with variance
\begin{align*}
\frac{\varsigma^2_1(H)}{|\bar\sigma|^4}\sum_{i,k=1}^m\sum_{j=1}^{\tilde m}
\bar{\sigma}^{2}_{i,j}\bar{\sigma}^{2}_{k,j}  + \frac{2\varsigma^2_1(H)}{|\bar\sigma|^4}  \sum_{j<q} \left(\sum^m_{i=1}\bar\sigma_{i,j}\bar\sigma_{i,q}\right)^2
 + 2 \lim_{n \to \infty} \cov \left(A_{n,T} ,B_{n,T} \right).
\end{align*}
The first two summands are precisely $\varsigma_{\star}^2(H)$, and it
therefore remains only to show that $A_{n,T}$ and $B_{n,T}$ are, at
least asymptotically, uncorrelated. To demonstrate this, observe that
$E (B_{n,T})=0$ (in comparison, $E(A_{n,T})=-\frac{1}{\sqrt{n}}$), and as a consequence,
\begin{align*}
\cov \left(A_{n,T}, B_{n,T}\right) &= E\left(A_{n,T}B_{n,T}\right) \nonumber\\
&= E\left(\left(A_{n,T}+\sqrt{n}\right)B_{n,T}\right) \nonumber\\
&= \frac{n}{|\bar\sigma|^4} E \left(\left(\sum^{m}_{i=1}\sum^{\tilde m}_{j=1}\bar\sigma^{2}_{i,j}V^{j,j}_{n,T}\right)\left(\sum^{m}_{i=1}\sum_{j\neq q}\bar\sigma_{i,j}\bar\sigma_{i,q}V^{j,q}_{n,T}\right)\right) \nonumber\\
&= 0,
\end{align*}
since each summand in the product (recall that the $V$ terms are
themselves sums) includes a filtered observation of a component of
$W^H$ that is independent from the others in that summand. This concludes the proof.
\end{proof}

We now apply Lemma \ref{L:SpecialConvergence} to deduce an idealized version of Theorem \ref{T:ConvergenceForHatH1} which we state in Lemma \ref{L:SpecialConvergenceForHatH1}; namely, we imagine that the data is sampled not from $\xe$ but rather from $\sqrt\epsilon\bar\sigma W^H$. Note that the $\epsilon$-dependence in the estimator and in the ideal data compensate one another and the asymptotic regime of interest is therefore once again simply $n\to\infty$.

\begin{lemma}\label{L:SpecialConvergenceForHatH1}
With notation as in the statement of Theorem \ref{T:ConvergenceForHatH1}, we have that
\begin{align*}
\hat{H}^\epsilon_1(\{\sqrt\epsilon\bar\sigma W^H_\tnk\}_{k=0}^n)&\rightarrow H \textrm{ in probability as } n\rightarrow\infty \textrm{ and } \nonumber\\
2\sqrt{n}\ln\left(\frac{n}{T}\right)\left(\hat{H}^\epsilon_1(\{\sqrt\epsilon\bar\sigma W^H_\tnk\}_{k=0}^n)- H\right)&\rightarrow \mathcal{N}(0,\varsigma^{2}_{\star}(H)) \textrm{ in distribution as }n\rightarrow\infty.
\end{align*}
Note that the value of $\hat{H}^\epsilon_1(\{\sqrt\epsilon\bar\sigma W^H_\tnk\}_{k=0}^n)$ does not depend on $\epsilon$ because the dependence in the estimator and the dependence in the ideal sampled data compensate one another.
\end{lemma}

\begin{proof}[Proof of Lemma \ref{L:SpecialConvergenceForHatH1}]
It suffices to show that
\begin{align*}
\frac{\phi_{n,T}(\hat{H}^\epsilon_1(\{\sqrt\epsilon\bar\sigma W^H_\tnk\}_{k=0}^n))}{\phi_{n,T}(H)}&\xrightarrow{prob.}1,\\
\sqrt{n}\left(\frac{\phi_{n,T}(\hat{H}^\epsilon_1(\{\sqrt\epsilon\bar\sigma W^H_\tnk\}_{k=0}^n))}{\phi_{n,T}(H)}-1\right)&\xrightarrow{\mathscr{D}}\mathcal{N}(0,\varsigma^{2}_{\star}(H)),
\end{align*}
whence the claim follows by reasoning as in \cite{kubiliusskorniakov, kubiliusskorniakovMelichov}.

If $\phi^{-1}_{n,T}$ were a right inverse of $\phi_{n,T}$, then this would be precisely the content of Lemma \ref{L:SpecialConvergence}, for one could write
\begin{align}
\phi_{n,T}\left(\hat{H}^\epsilon_1(\{\sqrt\epsilon\bar\sigma W^H_\tnk\}_{k=0}^n)\right) &= \left(\phi_{n,T}\circ\phi^{-1}_{n,T}\right)\left(\frac{1}{n|\bar\sigma|^2}\sum_{k=2}^{n}\left|\bar\sigma \Deltatwonk W^H \right|^{2}\right)\nonumber\\
&= \frac{1}{n|\bar\sigma|^2}\sum_{k=2}^{n}\left|\bar\sigma \Deltatwonk W^H \right|^{2}\label{Eq:Invertible}.
\end{align}

Of course, one actually has $\phi_{n,T}\circ\phi^{-1}_{n,T} = \min\{\cdot, 3\}$ (see Remark \ref{R:phi}). However, equation (\ref{Eq:VAlmostSureLimit}) implies in particular that for each $1 \leq j \leq \tilde m$,
\begin{align*}
\frac{1}{n}\sum^n_{k=2}|\Deltatwonk W^{H,j}|^2 &\xrightarrow {a.s.} 0,
\end{align*}
and hence also
\begin{align*}
0\leq\frac{1}{n|\bar\sigma|^2}\sum^n_{k=2}|\bar\sigma \Deltatwonk W^H|^2 &\leq \frac{\|\bar\sigma\|^2}{|\bar\sigma|^2}\sum^{\tilde m}_{j=1}\left( \frac{1}{n}\sum^n_{k=2}|\Deltatwonk W^{H,j}|^2 \right) \xrightarrow {a.s.} 0.
\end{align*}
Thus, one sees that for almost every realization of the random state, it is true that for $n$ sufficiently large the equality (\ref{Eq:Invertible}) is valid. Having made this observation the claim of the lemma follows by appeal to Lemma \ref{L:SpecialConvergence}.

\end{proof}
With the above lemmas and results at hand, we are now ready to present the proof of Theorem \ref{T:ConvergenceForHatH1}.
\begin{proof}[Proof of Theorem \ref{T:ConvergenceForHatH1}]
It suffices to show that
\begin{align*}
\frac{\phi_{n,T}(\hat{H}^\epsilon_1(\{\xe_\tnk\}_{k=0}^n))}{\phi_{n,T}(H)}&\xrightarrow{prob.}1,\\
\sqrt{n}\left(\frac{\phi_{n,T}(\hat{H}^\epsilon_1(\{\xe_\tnk\}_{k=0}^n))}{\phi_{n,T}(H)}-1\right)&\xrightarrow{\mathscr{D}}\mathcal{N}(0,\varsigma^{2}_{\star}(H)),
\end{align*}
and the claim follows by the same arguments as in \cite{kubiliusskorniakov,kubiliusskorniakovMelichov}. To do so, we use Lemmas \ref{L:FilteredDifferencesSquare} and \ref{L:SpecialConvergenceForHatH1}. Indeed, combining the approximation of Lemma \ref{L:FilteredDifferencesSquare} with the reasoning of the proof of Lemma \ref{L:SpecialConvergenceForHatH1}, one sees that
\begin{align*}
\phi_{n,T}\left(\hat{H}^\epsilon_1(\{\xe_\tnk\}_{k=0}^n)\right) &= \left(\phi_{n,T}\circ\phi^{-1}_{n,T}\right)\left(\frac{1}{n\epsilon|\bar\sigma|^2}\sum_{k=2}^{n}\left| \Deltatwonk \xe \right|^{2}\right)\\
&= \frac{1}{n|\bar\sigma|^2}\sum_{k=2}^{n}\left|\bar\sigma \Deltatwonk W^H \right|^{2} + o(n^{-2H}),
\end{align*}
where the little-$o$ is understood in probability as $n\to\infty$ and $\varepsilon:=(\epsilon,\eta)\to0$. Therefore,
\begin{align*}
\frac{\phi_{n,T}(\hat{H}^\epsilon_1(\{\xe_\tnk\}_{k=0}^n))}{\phi_{n,T}(H)}&=\frac{\phi_{n,T}(\hat{H}^\epsilon_1(\{\sqrt\epsilon\bar\sigma W^H_\tnk\}_{k=0}^n))}{\phi_{n,T}(H)} + o(1),
\end{align*}
whence the claim follows by appeal to Lemma \ref{L:SpecialConvergenceForHatH1}.
\end{proof}
Whenever one does not have knowledge of the magnitude of the noise
$\epsilon>0$, which is required for using the estimator
$\hat{H}_1^{\epsilon}$ and applying Theorem
\ref{T:ConvergenceForHatH1}, we introduce a second estimator of $H$
for which knowledge of the magnitude of the noise
$\epsilon>0$ is not required. This is the objcet of Theorem
\ref{T:ConvergenceForHatH2} below.
\begin{theorem}\label{T:ConvergenceForHatH2}
Assume Condition \ref{c:regularity} and Condition \ref{c:recurrencebasic} and, if $\sigma$ is nonconstant, assume that for some $\alpha > 0$ and $\beta \geq 2$ we also have Condition \ref{c:recurrence}-$(\alpha, \beta)$. Suppose that $n \to \infty$ and $\varepsilon := (\epsilon, \eta) \to 0$ in such a way that
\begin{enumerate}
	\item there is a $\rho_1 > 0$ such that $\eta^{-1} = O((\epsilon n^{2-2H})^{\rho_1})$ and
	\item if $\sigma$ is nonconstant, there is a $\rho_2 > 0$ such that $\eta = O(n^{-2-\rho_2})$.
\end{enumerate}
For $n>T$, given a random sample $\{\xe_\tnk\}_{k=0}^{2n}$, define the estimate
\begin{align}
\hat{H}_2(\{\xe_\tnk\}_{k=0}^{2n})
&:=\frac{1}{2}-\frac{1}{2\ln 2}\ln\left(\frac{\sum_{k=2}^{2n}\left|\Delta^{(2)}_{2n,k}\xe\right|^{2}}
{\sum_{k=2}^{n}\left|\Deltatwonk\xe\right|^{2}}\right).\label{Eq:HatHTwoDefinition}
\end{align}
Note that the filtered increments in the denominator are taken with double the spacing relative to those in the numerator.

We have that
\begin{equation*}
\hat{H}_2(\{\xe_\tnk\}_{k=0}^{2n})\rightarrow H \textrm{ in
                                    probability as }n\rightarrow\infty
                   \textrm{ and } \varepsilon :=
                                    (\epsilon, \eta) \to 0
\end{equation*}
                                    and
\begin{align*}
2\ln 2 \sqrt{n}\left(\hat{H}_2(\{\xe_\tnk\}_{k=0}^{2n})- H\right)&\rightarrow \mathcal{N}(0,\varsigma^{2}_{\star\star}(H)) \textrm{ in distribution as }n\rightarrow\infty \textrm{ and }\varepsilon := (\epsilon,\eta) \to 0,
\end{align*}
where the variance is given by
\[
\varsigma^{2}_{\star\star}(H):=\left(\frac{3}{2}\varsigma^{2}_{1}(H)-2\varsigma^{2}_{2}(H)\right)\left(\frac{1}{|\bar\sigma|^4} \sum_{i,k=1}^m\sum_{j,q=1}^{\tilde{m}}  \bar\sigma_{i, j}
\bar\sigma_{i, q} \bar\sigma_{k, j} \bar\sigma_{k,q}\right)
\]
where $m$ is the dimension of the slow process $X$, $\tilde m$ is the dimension of the noise $W^H$, $\varsigma^{2}_{1}(H)$ is as in Theorem \ref{T:ConvergenceForHatH1}, and
\begin{align*}
\varsigma^{2}_{2}(H)&:= \sum_{j \in \mathbb{Z}} \tilde{\rho}^2(j;H)
\end{align*}
with
\begin{align*}
 \tilde{\rho}(j;H) &:= \frac{-|j-3|^{2H} + 2|j-2|^{2H} + |j-1|^{2H} - 4|j|^{2H}+|j+1|^{2H}+2|j+2|^{2H}-|j+3|^{2H}}{2(4-2^{2H})2^H}.
\end{align*}
\end{theorem}

\begin{proof}[Proof of Theorem \ref{T:ConvergenceForHatH2}] The proof is very similar to that already given for the first estimator and so we only indicate the direction of it.

By the approximation of Lemma \ref{L:FilteredDifferencesSquare} and with the notation of the proof of Lemma \ref{L:SpecialConvergence}, the behavior of the estimator is asymptotically equivalent in probability to
\begin{align*}
H - \frac{1}{2 \ln 2}\log\left(\frac{\sum^{m}_{i=1}\sum^{\tilde m}_{j,q=1}\bar\sigma_{i,j}\bar\sigma_{i,q}V^{j,q}_{2n,T}}{\sum^{m}_{i=1}\sum^{\tilde m}_{j,q=1}\bar\sigma_{i,j}\bar\sigma_{i,q}V^{j,q}_{n,T}}\right),
\end{align*}
and so one is led to consider the limiting behavior of this quantity.

One calculates that
\begin{align*}
\frac{1}{|\bar\sigma|^2}\sum^{m}_{i=1}\sum^{\tilde m}_{j,q=1}\bar\sigma_{i,j}\bar\sigma_{i,q}V^{j,q}_{n,T} \xrightarrow{a.s.} 1
\end{align*}
(the same of course for $2n$ in place of $n$) and that moreover
\begin{align*}
\sqrt{n}\begin{pmatrix}
\frac{1}{|\bar\sigma|^2}\sum^{m}_{i=1}\sum^{\tilde m}_{j,q=1}\bar\sigma_{i,j}\bar\sigma_{i,q}V^{j,q}_{n,T} - 1 \\
\frac{1}{|\bar\sigma|^2}\sum^{m}_{i=1}\sum^{\tilde m}_{j,q=1}\bar\sigma_{i,j}\bar\sigma_{i,q}V^{j,q}_{2n,T} - 1
\end{pmatrix} \xrightarrow{\mathscr{D}} \mathcal{N} \left(\begin{pmatrix}
0 \\
0
\end{pmatrix},   \begin{pmatrix}
   \varsigma_1^2(H)  &  \varsigma_2^2(H) \\
   \varsigma_2^2(H)  &  \varsigma_1^2(H) / 2
  \end{pmatrix}  \left(\frac{1}{|\bar\sigma|^4} \sum_{i,k=1}^m\sum_{j,q=1}^{\tilde{m}}  \bar\sigma_{i, j}
\bar\sigma_{i, q} \bar\sigma_{k, j} \bar\sigma_{k,q}\right)\right).
\end{align*}
The result follows from this by the delta method.
\end{proof}

\section{Trajectory fitting estimator for drift estimation}\label{S:TFE1}
Suppose that the drift coefficient $c$ in (\ref{Eq:ModelSystem}) depends upon an unknown parameter $\theta$ that lies in some bounded, convex, open subset $\Theta$ of some Euclidean space. We write
\begin{equation*}
\begin{cases} d\xa_t =
c_\theta(\xa_t, \ye_t) dt + \sqrt\epsilon \sigma(\ye_t) dW^H_t \\ d\ye_t = \frac{1}{\eta} f(\ye_t) dt
+ \frac{1}{\sqrt\eta} \tau(\ye_t) dB_t \\
\xa_0=x_0\in\mathcal{X}, \hspace{1pc} \ye_0=y_0\in\mathcal{Y}.
\end{cases}
\end{equation*}

\begin{definition}\label{zdefinition}
For $\theta\in\Theta$ and $0 \leq s \leq t \leq T$, let $Z^\theta(t, s)$ denote the matrix-valued solution to the equation
\begin{align*}
\frac{dZ^\theta(t, s)}{dt} = (\nabla_x \bar c_\theta)(\bar X^\theta_t) Z^\theta(t, s), \hspace{1pc} Z^\theta(s, s) = 1_m,
\end{align*}
where $1_m$ denotes the $m \times m$ identity matrix. By \cite[Proposition 2.14]{gilbargtrudinger}, the continuity of $t \mapsto (\nabla_x \bar c_\theta)(\bar X^\theta_t)$ on $(0, T]$ guarantees the semigroup relations $Z^\theta(t, s)Z^\theta(s, r) = Z^\theta(t, r)$ and the invertibility of $Z^\theta(t, s)$.
\end{definition}
A very simple approach to estimating $\theta$ is to choose a value
$\theta\in\bar\Theta$ that brings the vector $\{\bar
X^\theta_\tnk\}_{k=1}^n$ into nearest possible agreement with the
vector $\{\xa_\tnk\}_{k=1}^n$ of observations. We impose an
identifiability condition to guarantee that the choice is unique. Let us take as our contrast function the squared Euclidean distance
\begin{align*}
U(\theta; \{x_\tnk\}_{k=1}^n)&:=\Sigma_{k=1}^n \left| x_\tnk - \bar X^\theta_\tnk \right|^2.
\end{align*}

\begin{remark}
There are many other reasonable contrast functions that one could choose to use here. For example, with $Z^\theta$ as in Definition \ref{zdefinition}, recalling that $t^n_0=0$, one can define the alternative contrast function $U'(\theta; \{x_\tnk\}_{k=1}^n):=\Sigma_{k=1}^n \left| (x_\tnk - \bar X^\theta_\tnk) - Z^\theta(\tnk, \tnkm) (x_\tnkm - \bar X^\theta_\tnkm) \right|^2$. It is then easy to derive results analogous to the ones derived in this section for $U(\theta; \{x_\tnk\}_{k=1}^n)=\Sigma_{k=1}^n \left| x_\tnk - \bar X^\theta_\tnk \right|^2$, although it should be noted that a stronger identifiability condition may be required. We have chosen to work with the particular choice $U(\theta; \{x_\tnk\}_{k=1}^n)$ because the resulting estimator is perhaps among the simplest to analyze, implement, and compute, while exhibiting at the same time general stability in numerical behavior (see Section \ref{S:NumericalExamples}).
\end{remark}
\begin{remark}
Taking the $L^2$ distance requires that $T\beta\gamma\sup_{y\in\mathcal{Y}}\|\tau(y)\|^2<\alpha$ so that one may take $p=2$ in \cite[Theorem 1]{fluctuations_fbm_multiscale}, but makes it easier to deduce asymptotic normality of the resulting estimates by studying the Hessian of the contrast function.
\end{remark}
Finally, we define the mean-square Trajectory-Fitting Estimator
\begin{align}
\hat\theta_{TFE}(\{x_\tnk\}_{k=1}^n):=\arg\min_{\theta\in\bar\Theta}U(\theta; \{x_\tnk\}_{k=1}^n).\label{Eq:TFEDefinition}
\end{align}

\subsection{Finite sample calculations}\label{S:FixedNoObs}

In this section, we assume that we have a fixed number of observations, $n$, of the slow process, say $\{x_\tnk\}_{k=1}^n$ observed at given discrete times $\{\tnk\}_{k=1}^n\subset [0,T]$. We impose an identifiability condition to guarantee that for each $\theta_0\in\Theta$, $U(\theta; \{\bar X^{\theta_0}\}_{k=1}^n)$ is uniquely minimized over $\theta\in\bar\Theta$ at $\theta=\theta_0$.

\begin{condition}\label{first_identifiability_condition} (First Identifiability Condition)
For a given choice of $n\in\mathbb{N}$ we shall write Condition \ref{first_identifiability_condition}-($n$) to mean that for any $(\theta_1, \theta_2) \in \bar\Theta^2$, $\{\bar X^{\theta_1}_\tnk\}_{k=1}^n = \{\bar X^{\theta_2}_\tnk\}_{k=1}^n$ if and only if $\theta_1 = \theta_2$.
\end{condition}
We are going to show that this estimator is consistent and
asymptotically normal. Let us start with the consistency, which follows
from a modulus-of-continuity argument.
\begin{theorem}\label{consistencytfe} (Consistency of the Trajectory-Fitting Estimator)
Let $n\in\mathbb{N}$ be given. Assume Conditions \ref{c:regularity}, \ref{c:recurrence}-$(\alpha,\beta,\gamma)$ and \ref{first_identifiability_condition}-($n$). Assume also that $T\beta\gamma\sup_{y\in\mathcal{Y}}\|\tau(y)\|^2<\alpha$ so that one may take $p=2$ in \cite[Theorem 1]{fluctuations_fbm_multiscale}. For any $\theta_0\in\Theta$, $H_0\in(1/2, 1)$, and $\zeta>0$,
\begin{align*}
\lim_{\varepsilon\to0}P\left(|\hat\theta_{TFE}(\{\xazero_\tnk\}_{k=1}^n)-\theta_0|>\zeta\right)=0.
\end{align*}
\end{theorem}

\begin{proof}[Proof of Theorem \ref{consistencytfe}]
Consider the modulus of continuity
\begin{align*}
w(\phi; \{x_\tnk\}_{k=1}^n):=\sup_{(\theta_1, \theta_2)\in\Theta^2;|\theta_1-\theta_2|\leq\phi}|U(\theta_1; \{x_\tnk\}_{k=1}^n)-U(\theta_2; \{x_\tnk\}_{k=1}^n)|
\end{align*}
defined for $\phi>0$. It is clear from \cite[Theorem 1]{fluctuations_fbm_multiscale} that as $\varepsilon\to0$, $w(\phi; \{\xa_\tnk\}_{k=1}^n)$ converges in probability uniformly in $\theta\in\bar\Theta$ and $\phi>0$ to $w(\phi; \{\bar X^\theta_\tnk\}_{k=1}^n)$. It is also clear by continuity of $\theta \mapsto \{\bar X^\theta_\tnk\}_{k=1}^n$ and the triangle inquality that $\lim_{\phi\to0}\sup_{\theta\in\bar\Theta}w(\phi; \{\bar X^\theta_\tnk\}_{k=1}^n) = 0$. The claim of the theorem follows by \cite[Theorem 3.2.8]{dc:probability}.
\end{proof}
The next result establishes the asymptotic normality of our estimator.
\begin{theorem}\label{normalitytfe} (Asymptotic Normality of the Trajectory-Fitting Estimator)
Let $n\in\mathbb{N}$ be given. Assume Conditions \ref{c:regularity}, \ref{c:recurrence}-$(\alpha,\beta,\gamma)$ and \ref{first_identifiability_condition}-($n$). Assume also that $T\beta\gamma\sup_{y\in\mathcal{Y}}\|\tau(y)\|^2<\alpha$ so that one may take $p=2$ in \cite[Theorem 1]{fluctuations_fbm_multiscale}. For any $\theta_0\in\Theta$, $H_0\in(1/2, 1)$,  we have that
$\frac{1}{\sqrt{\epsilon}}\left(\hat\theta_{TFE}(\{\xazero_\tnk\}_{k=1}^n)-\theta_0\right)$ converges in distribution as $\varepsilon:=(\epsilon, \eta) \to 0$ to the normal distribution $\mathcal{N}(0, M(\theta_0, H_0; n))$, where the variance is given by
\begin{align*}
M(\theta, H; n)
 &:= \left( \Sigma_{k=1}^n \left[ (\nabla_\theta\Big|_{\theta}\bar X^\cdot_\tnk)^T \nabla_\theta\Big|_{\theta}\bar X^\cdot_\tnk \right] \right)^{-1} \times \\
 & \hspace{4pc} \left( \Sigma_{j, k = 1}^n \left[ \left( \nabla_\theta\Big|_{\theta} \bar X^\cdot_\tnj \right)^T E\left[ \xia_\tnj(\xia_\tnk)^T \right] \left( \nabla_\theta\Big|_{\theta} \bar X^\cdot_\tnk \right) \right] \right) \hspace{1pc} \times \\
 &\hspace{4pc}  \left( \Sigma_{k=1}^n \left[ (\nabla_\theta\Big|_{\theta}\bar X^\cdot_\tnk)^T \nabla_\theta\Big|_{\theta}\bar X^\cdot_\tnk \right] \right)^{-1}.
\end{align*}
\end{theorem}
Before presenting the proof of this theorem, let us make some related remarks.
\begin{remark}\label{R:LimitingVariance}
For the limit as $n\to\infty$ of the limiting variance, one may write explicitly
\begin{align}
\lim_{n\to\infty} M(\theta, H; n) &= \left( \int^T_{t=0} \left[ (\nabla_\theta\Big|_{\theta}\bar X^\cdot_t)^T \nabla_\theta\Big|_{\theta}\bar X^\cdot_t \right] dt \right)^{-1} \times \nonumber\\
 & \hspace{4pc} \left( \int^T_{\tone=0} \int^T_{\ttwo=0} \left[ \left( \nabla_\theta\Big|_{\theta} \bar X^\cdot_\tone \right)^T E\left[ \xia_\tone(\xia_\ttwo)^T \right] \left( \nabla_\theta\Big|_{\theta} \bar X^\cdot_\ttwo \right) \right] d\ttwo d\tone \right) \hspace{1pc} \times \nonumber\\
 &\hspace{4pc}  \left( \int^T_{t=0} \left[ (\nabla_\theta\Big|_{\theta}\bar X^\cdot_t)^T \nabla_\theta\Big|_{\theta}\bar X^\cdot_t \right] dt \right)^{-1}.\nonumber
\end{align}
\end{remark}

\begin{remark}
For $t \in [0, T]$, one may write explicitly
\begin{align}
\nabla_\theta\Big|_{\theta}\bar X^\cdot_t &= \int^t_0 Z^\theta(t, s) (\nabla_\theta\bar c_\cdot)(\bar X^\theta_s)ds.\nonumber
\end{align}
\end{remark}

\begin{remark}
For $(\tone, \ttwo) \in [0,T]^2$, one may write explicitly
\begin{align}
E\left[ \xia_\tone(\xia_\ttwo)^T \right] &= \int^{\tone \wedge \ttwo}_{s=0} \left( Z^\theta(\tone, s) \lambda \Sigma^\theta_\Phi(\bar X^\theta_s) \right) \left( Z^\theta(\ttwo, s) \lambda \Sigma^\theta_\Phi(\bar X^\theta_s) \right)^T ds \nonumber\\
& + H(2H-1) \int^\tone_{\sone=0} \int^\ttwo_{\stwo=0} \left( Z^\theta(\tone, \sone) \bar\sigma \right) \left( Z^\theta(\ttwo, \stwo) \bar\sigma \right)^T | \sone - \stwo |^{2H-2} d\stwo d\sone.\nonumber
\end{align}
\end{remark}

\begin{proof}[Proof of Theorem \ref{normalitytfe}] Let us suppress the data $\{\xazero_\tnk\}_{k=1}^n$. By Taylor's theorem,
\begin{align*}
0 &= \frac{1}{\sqrt\epsilon}\nabla_\theta\Big|_{\hat\theta_{TFE}}U = \frac{1}{\sqrt\epsilon}\nabla_\theta\Big|_{\theta_0}U + \frac{1}{\sqrt\epsilon}(\hat\theta_{TFE}-\theta_0)^T \nabla^2_\theta\Big|_{\theta^\dagger}U,
\end{align*}
where $\theta^\dagger$ is an appropriately-chosen point on the line segment connecting $\hat\theta_{TFE}$ with $\theta_0$. Assuming the inverse exists, we may reexpress this as
\begin{align*}
\frac{1}{\sqrt\epsilon}(\hat\theta_{TFE}-\theta_0) &= \left(\nabla^2_\theta\Big|_{\theta^\dagger}U\right)^{-1} \left(-\frac{1}{\sqrt\epsilon}\nabla_\theta\Big|_{\theta_0}U\right)^T.
\end{align*}

Thus, it suffices to establish a limit in distribution of $-\frac{1}{\sqrt\epsilon}\nabla_\theta\Big|_{\theta_0}U$ and an invertible limit in probability of $\nabla^2_\theta\Big|_{\theta^\dagger}U$; the interested reader is referred to \cite[Section 3.3.4]{dc:probability} for a rigorous justification of this now-classical approach.

For the rescaled gradient of the contrast function, we have
\begin{align}
-\frac{1}{\sqrt\epsilon}\nabla_\theta\Big|_{\theta_0}U &= 2 \Sigma_{k=1}^n \left[ \frac{1}{\sqrt\epsilon} (\xazero_\tnk - \bar X^{\theta_0}_\tnk)^T \nabla_\theta\Big|_{\theta_0} \bar X^\cdot_\tnk \right],\label{Eq:GradientContrastFcn}
\end{align}
which converges in distribution to $ 2 \Sigma_{k=1}^n \left[ (\xiazero_\tnk)^T \nabla_\theta\Big|_{\theta_0} \bar X^\cdot_\tnk \right]$ by \cite[Theorem 2]{fluctuations_fbm_multiscale}, where $\xia$, the limit of the fluctuations, is given for $t\in[0,T]$ by the mixed stochastic integral (recall Theorem \ref{T:LimitBehavior})
\begin{align*}
\xia_t &= \lambda \int^t_0 Z^\theta(t, s) \Sigma_\Phi^\theta(\bar X_s)  d \tilde B_s + \int^t_0 Z^\theta(t, s) \bar \sigma d \tilde W^H_s,
\end{align*}
where $\tilde W^H$ is a fractional Brownian motion with Hurst parameter $H$ and $\tilde B$ is a standard Brownian motion independent from $\tilde W^H$.

Meanwhile, for the Hessian of the contrast function, we have
\begin{align}
\nabla^2_\theta\Big|_{\theta^\dagger}U &= 2 \Sigma_{k=1}^n \left[ (\nabla_\theta\Big|_{\theta^\dagger} \bar X^\cdot_\tnk)^T \nabla_\theta\Big|_{\theta^\dagger} \bar X^\cdot_\tnk \right] - 2 \Sigma_{k=1}^n \nabla_\theta\Big|_{\theta^\dagger}\left( (\xazero_\tnk - \bar X^{\theta^\dagger}_\tnk)^T \nabla_\theta\Big|_{\theta}\bar X^\cdot_\tnk \right)^T,\label{Eq:HessianContrastFcn}
\end{align}
which converges in probability to $2 \Sigma_{k=1}^n \left[ (\nabla_\theta\Big|_{\theta_0}\bar X^\cdot_\tnk)^T \nabla_\theta\Big|_{\theta_0}\bar X^\cdot_\tnk \right]$; to see this, recall that $\hat\theta_{TFE}$ converges in probability to $\theta_0$ by Theorem \ref{consistencytfe}, that $\xazero_t$ converges in probability to $\bar X^{\theta_0}_t$ uniformly in $t\in[0, T]$, and that $\nabla^2_\theta\Big|_\theta \bar X^\cdot_t$ is bounded uniformly in $\theta\in\Theta$ and $t\in[0, T]$.

Putting these together,
\begin{align}
\frac{1}{\sqrt\epsilon}(\hat\theta_{TFE}(\{\xazero_\tnk\}_{k=1}^n)-\theta_0) &\xrightarrow{\mathcal{D}} \left( \Sigma_{k=1}^n \left[ (\nabla_\theta\Big|_{\theta_0}\bar X^\cdot_\tnk)^T \nabla_\theta\Big|_{\theta_0}\bar X^\cdot_\tnk \right] \right)^{-1} \left( \Sigma_{k=1}^n \left[ (\xiazero_\tnk)^T \nabla_\theta\Big|_{\theta_0} \bar X^\cdot_\tnk \right] \right)^T,\nonumber
\end{align}
which is to say that
$\frac{1}{\sqrt\epsilon}(\hat\theta_{TFE}(\{\xazero_\tnk\}_{k=1}^n)-\theta_0)$
converges in distribution as $\varepsilon\to0$ to $\mathcal{N}(0,
M(\theta_0, H_0; n))$. This completes the proof of the theorem.
\end{proof}

\subsection{High-frequency regime for the trajectory fitting estimator}\label{S:HighFrequencyTFE1}
Here, in contrast with Subsection \ref{S:FixedNoObs}, where we assume that the number of observations $n$ is fixed, we consider the case where $n$ grows to infinity at the same time as $\epsilon,\eta$ are taken to zero. For notational convenience, let us write $\Delta:=T/n$ for the sampling interval.

\begin{theorem}\label{consistencytfe_highFrequency} (Consistency of the Trajectory-Fitting Estimator as $(\epsilon+\eta+\Delta)\to0$)
Assume Conditions \ref{c:regularity}, \ref{c:recurrence}-$(\alpha,\beta,\gamma)$ and \ref{first_identifiability_condition}-($n$). Assume also that $T\beta\gamma\sup_{y\in\mathcal{Y}}\|\tau(y)\|^2<\alpha$ so that one may take $p=2$ in \cite[Theorem 1]{fluctuations_fbm_multiscale}. Assume that $\epsilon$, $\eta$, and $\Delta$ all go to zero in such a way that $\frac{\sqrt{\epsilon}+\sqrt{\eta}}{\Delta}\rightarrow 0$. For any $\theta_0\in\Theta$, $H_0\in(1/2, 1)$, and $\zeta>0$,
\begin{align*}
\lim_{(\epsilon+\eta+\Delta)\to0}P\left(|\hat\theta_{TFE}(\{\xazero_\tnk\}_{k=1}^n)-\theta_0|>\zeta\right)=0.
\end{align*}
\end{theorem}

\begin{proof}[Proof of Theorem \ref{consistencytfe_highFrequency}]
The proof of this result is very similar to that of Theorem \ref{consistencytfe}. The main difference is that we would like to show that the assumed relationship among $\epsilon$, $\eta$, and $\Delta$ is sufficient to conclude that, for any $\zeta>0$,
\begin{align*}
\lim_{(\epsilon+\eta+\Delta)\to0}P\left(\sup_{\theta_1,\theta_2\in\Theta}|U(\theta_2; \{X^{\epsilon,\theta_1}_{t_k}\}_{k=1}^n)- U(\theta_2; \{\bar X^{\theta_1}_{t_k}\}_{k=1}^n))|>\zeta\right)=0.
\end{align*}

To this end, we have that
\begin{align}
U(\theta_2; \{X^{\epsilon,\theta_1, H_{0}}_{t_k}\}_{k=1}^n)- U(\theta_2; \{\bar X^{\theta_1}_{t_k}\}_{k=1}^n))&=\sum_{k=1}^{n}\left(X^{\epsilon,\theta_1, H_{0}}_{t_k}-\bar X^{\theta_1}_{t_k}\right)\left(X^{\epsilon,\theta_1, H_{0}}_{t_k}-2 \bar X^{\theta_2}_{t_k}+\bar X^{\theta_2}_{t_k}\right)\nonumber\\
&=(\sqrt{\epsilon}+\sqrt{\eta})\sum_{k=1}^{n}\frac{X^{\epsilon,\theta_1, H_{0}}_{t_k}-\bar X^{\theta_1}_{t_k}}{\sqrt{\epsilon}+\sqrt{\eta}}\left(X^{\epsilon,\theta_1, H_{0}}_{t_k}-2 \bar X^{\theta_2}_{t_k}+\bar X^{\theta_2}_{t_k}\right).\nonumber
\end{align}

From this we conclude that this term vanishes in $L^{1}$ by Theorem \ref{T:LimitBehavior} together with the assumption that $\frac{\sqrt{\epsilon}+\sqrt{\eta}}{\Delta}\rightarrow 0$. The rest of the proof follows the proof of Theorem \ref{consistencytfe}.
\end{proof}

Let us now consider the asymptotic normality statement in the high-frequency regime. Recall that we stated in Remark \ref{R:LimitingVariance} the limit, as $n\to\infty$, of the limiting variance for a fixed number $n$ of observed data. The same expression is obtained as the variance of a joint limit provided that $\epsilon, \eta, \Delta$ satisfy the same asymptotic relationship as in Theorem \ref{consistencytfe_highFrequency}.
\begin{theorem}\label{normalitytfe_highFrequency} (Asymptotic Normality of the Trajectory-Fitting Estimator  as $(\epsilon+\eta+\Delta)\to0$)
Assume Conditions \ref{c:regularity}, \ref{c:recurrence}-$(\alpha,\beta,\gamma)$ and \ref{first_identifiability_condition}-($n$). Assume also that $T\beta\gamma\sup_{y\in\mathcal{Y}}\|\tau(y)\|^2<\alpha$ so that one may take $p=2$ in \cite[Theorem 1]{fluctuations_fbm_multiscale}. Assume that $\epsilon$, $\eta$, and $\Delta$ all go to zero in such a way that $\frac{\sqrt{\epsilon}+\sqrt{\eta}}{\Delta}\rightarrow 0$. For any $\theta_0\in\Theta$, $H_0\in(1/2, 1)$,  we have that
$\frac{1}{\sqrt{\epsilon}}\left(\hat\theta_{TFE}(\{\xazero_\tnk\}_{k=1}^n)-\theta_0\right)$ converges in distribution as $\epsilon,\eta,\Delta\to 0$ to the normal distribution $\mathcal{N}(0, \bar{M}(\theta_0, H_0))$, where the variance is given by
\begin{align*}
\bar{M}(\theta, H) &:=
\left( \int^T_{t=0} \left[ (\nabla_\theta\Big|_{\theta}\bar X^\cdot_t)^T \nabla_\theta\Big|_{\theta}\bar X^\cdot_t \right] dt \right)^{-1} \times \\
 & \hspace{4pc} \left( \int^T_{\tone=0} \int^T_{\ttwo=0} \left[ \left( \nabla_\theta\Big|_{\theta} \bar X^\cdot_\tone \right)^T E\left[ \xia_\tone(\xia_\ttwo)^T \right] \left( \nabla_\theta\Big|_{\theta} \bar X^\cdot_\ttwo \right) \right] d\ttwo d\tone \right) \hspace{1pc} \times \\
 &\hspace{4pc}  \left( \int^T_{t=0} \left[ (\nabla_\theta\Big|_{\theta}\bar X^\cdot_t)^T \nabla_\theta\Big|_{\theta}\bar X^\cdot_t \right] dt \right)^{-1}.
\end{align*}
\end{theorem}

\begin{proof}[Proof of Theorem \ref{normalitytfe_highFrequency}]
The proof of this theorem follows closely that of Theorem
\ref{normalitytfe} with the additional element of accounting for the
limit as $\Delta\to 0$. Below, we only comment on the differences. In particular, by (\ref{Eq:GradientContrastFcn}) we have that
\begin{align*}
-\frac{\Delta}{\sqrt\epsilon}\nabla_\theta\Big|_{\theta_0}U &= 2 \Delta\Sigma_{k=1}^n \left[ \frac{1}{\sqrt\epsilon} (\xazero_\tnk - \bar X^{\theta_0}_\tnk)^T \nabla_\theta\Big|_{\theta_0} \bar X^\cdot_\tnk \right],
\end{align*}
which converges in distribution to $ 2 \int_{0}^{T} \left[ (\xi^{\theta_0,H_0}_t)^T \nabla_\theta\Big|_{\theta_0} \bar X^\cdot_t \right]dt$ by \cite[Theorem 2]{fluctuations_fbm_multiscale}, where, as before, $\xia$ is the limit of the fluctuations,  given for $t\in[0,T]$ by the mixed stochastic integral
\begin{align}
\xia_t &= \lambda \int^t_0 Z^\theta(t, s) \Sigma_\Phi^\theta(\bar X_s)  d \tilde B_s + \int^t_0 Z^\theta(t, s) \bar \sigma d \tilde W^H_s.\nonumber
\end{align}

By (\ref{Eq:HessianContrastFcn}) we have that
\begin{align*}
\Delta\nabla^2_\theta\Big|_{\theta^\dagger}U &= 2 \Delta\Sigma_{k=1}^n \left[ (\nabla_\theta\Big|_{\theta^\dagger} \bar X^\cdot_\tnk)^T \nabla_\theta\Big|_{\theta^\dagger} \bar X^\cdot_\tnk \right] - 2 \Delta\Sigma_{k=1}^n \nabla_\theta\Big|_{\theta^\dagger}\left( (\xazero_\tnk - \bar X^{\theta^\dagger}_\tnk)^T \nabla_\theta\Big|_{\theta}\bar X^\cdot_\tnk \right)^T,
\end{align*}
which converges in probability to $2 \int_{0}^{T} \left[
  (\nabla_\theta\Big|_{\theta_0}\bar X^\cdot_t)^T
  \nabla_\theta\Big|_{\theta_0}\bar X^\cdot_t \right]dt$. To see that
the second term vanishes, we used the fact that  $\hat\theta_{TFE}$
converges in probability to $\theta_0$ by Theorem
\ref{consistencytfe_highFrequency}, that $\xazero_t$ converges in
probability to $\bar X^{\theta_0}_t$ uniformly in $t\in[0, T]$, and
that $\nabla^2_\theta\Big|_\theta \bar X^\cdot_t$ is bounded uniformly
in $\theta\in\Theta$ and $t\in[0, T]$. The rest of the proof follows that of Theorem \ref{normalitytfe}.
\end{proof}

\section{A Contrast Estimator Based on the Likelihood of an Approximate Model}\label{S:DriftEstimationMLE}

In this section we present an alternative estimator for the unknown parameter $\theta$ which is based on the principle of maximum likelihood applied to an approximate model motivated by the fluctuations approximation given precisely in Theorem \ref{T:LimitBehavior}. The advantage of this estimator is that it has smaller asymptotic variance than the one presented in Section \ref{S:TFE1}, at least when the Hurst index is known, as we demonstrate in the proof of Theorem \ref{T:VarianceComparisons}. One disadvantage, however, which is clear from the formulation (\ref{Eq:MCE}), is that it is computationally challenging to implment, in that one must invert a matrix which is quite large in typical cases. Nevertheless this estimator is of theoretical interest as well as practical interest when the computational challenges can be met. Also, note that the construction of this estimator involves a parameter which is ideally chosen to coincide with the Hurst index. However, it turns out that one has consistency and asymptotic normality even when the chosen parameter and the Hurst index do not coincide.

Recall that \cite[Theorem 2]{fluctuations_fbm_multiscale} states that the fluctuations process $\frac{1}{\sqrt\epsilon}\left(\xa - \bar X^\theta\right)$ converges in distribution to $\xia$, where for $t\in[0,T]$,
\begin{align*}
\xia_t &= \lambda \int^t_0 Z^\theta(t, s) \Sigma_\Phi^\theta d \tilde B_s + \int^t_0 Z^\theta(t, s) \bar \sigma d \tilde W^H_s,
\end{align*}
where $Z^\theta$ is as in Definition \ref{zdefinition}, $\tilde W^H$ is a fractional Brownian motion with Hurst parameter $H$, and $\tilde B$ is a standard Brownian motion independent from $\tilde W^H$.

Let $\oplus$ denote the vector concatenation sum and $\otimes$ the vector outer product. For each $n\in\mathbb{N}$, $\xian:=\oplus_{k=1}^n \xia_\tnk$ is a centered Gaussian vector with covariance matrix $\Xian:=E\left[\xian \otimes \xian\right]$.

The convergence in distribution of the fluctuations in \cite[Theorem 2]{fluctuations_fbm_multiscale} may be understood to mean that, in an appropriate sense, the vector of observations $\oplus_{k=1}^n \xa_\tnk$ is asymptotically Gaussian with mean $\oplus_{k=1}^n \bar X^{\theta}_\tnk$ and covariance matrix $\epsilon \hspace{0.1pc} \Xian$.

The form of the likelihood for this approximation suggests that it is reasonable to estimate $\theta$ by minimizing the contrast function
\begin{align*}
\tilde U^\epsilon(\theta; H, \{x_\tnk\}_{k=1}^n) := \epsilon \log \det ( \Xian ) + \left( \oplus_{k=1}^n ( x_\tnk - \bar X^\theta_\tnk ) \right)^T \left( \Xian \right)^{-1} \left( \oplus_{k=1}^n ( x_\tnk - \bar X^\theta_\tnk ) \right)
\end{align*}
or more simply
\begin{align*}
\tilde U(\theta; H, \{x_\tnk\}_{k=1}^n) := \left( \oplus_{k=1}^n ( x_\tnk - \bar X^\theta_\tnk ) \right)^T \left( \Xian \right)^{-1} \left( \oplus_{k=1}^n ( x_\tnk - \bar X^\theta_\tnk ) \right).
\end{align*}

We define the Minimum Contrast Estimator with parameter $\mathcal{H}$
\begin{align}
\hat\theta^{\mathcal{H}}_{MCE}\left( \{x_\tnk\}_{k=1}^n \right) &:= \arg\min_{\theta\in\bar\Theta} \tilde U(\theta; \mathcal{H}, \{x_\tnk\}_{k=1}^n).\label{Eq:MCE}
\end{align}

The estimator turns out to be consistent and asymptotically normal even if $\mathcal{H} \neq H_0$, i.e., even if the Hurst index is not correctly specified.
\begin{theorem} (Consistency of the Minimum Contrast Estimator)
Let $n\in\mathbb{N}$ be given. Assume Conditions \ref{c:regularity}, \ref{c:recurrence}-$(\alpha, \beta, \gamma)$, and \ref{first_identifiability_condition}-($n$). Assume also that $T\beta\gamma\sup_{y\in\mathcal{Y}}\|\tau(y)\|^2<\alpha$ so that one may take $p=2$ in \cite[Theorem 1]{fluctuations_fbm_multiscale}. For any $\theta_0\in\Theta$, $(H_0, \mathcal{H})\in(1/2, 1)^2$, and $\zeta>0$,
\begin{align*}
\lim_{\varepsilon\to0}P\left(|\hat\theta^\mathcal{H}_{MCE}(\{\xazero_\tnk\}_{k=1}^n)-\theta_0|>\zeta\right)=0.
\end{align*}
\end{theorem}

\noindent\textit{Proof.} Consider the modulus of continuity
\begin{align*}
\tilde w(\phi; \mathcal{H}, \{x_\tnk\}_{k=1}^n):=\sup_{(\theta_1, \theta_2)\in\Theta^2;|\theta_1-\theta_2|\leq\phi}|\tilde U(\theta_1; \mathcal{H}, \{x_\tnk\}_{k=1}^n)-U(\theta_2; \mathcal{H}, \{x_\tnk\}_{k=1}^n)|
\end{align*}
defined for $\phi>0$. It is clear from \cite[Theorem 1]{fluctuations_fbm_multiscale} that as $\varepsilon\to0$, $\tilde w(\phi; \mathcal{H}, \{X^{\varepsilon,\theta,H_0}_\tnk\}_{k=1}^n)$ converges in probability uniformly in $\theta\in\bar\Theta$ and $\phi>0$ to $\tilde w(\phi; \mathcal{H}, \{\bar X^\theta_\tnk\}_{k=1}^n)$. It is also clear by continuity of $\theta \mapsto \{\bar X^\theta_\tnk\}_{k=1}^n$ and the triangle inquality that $\lim_{\phi\to0}\sup_{\theta\in\bar\Theta} \tilde w(\phi; \mathcal{H}, \{\bar X^\theta_\tnk\}_{k=1}^n) = 0$. The claim of the theorem follows by \cite[Theorem 3.2.8]{dc:probability}.

\qed

\begin{theorem} \label{normalitymce} (Asymptotic Normality of the Minimum-Contrast Estimator)

Let $n\in\mathbb{N}$ be given. Assume Conditions \ref{c:regularity}, \ref{c:recurrence}-$(\alpha, \beta, \gamma)$, and \ref{first_identifiability_condition}-($n$). Assume also that $T\beta\gamma\sup_{y\in\mathcal{Y}}\|\tau(y)\|^2<\alpha$ so that one may take $p=2$ in \cite[Theorem 1]{fluctuations_fbm_multiscale}. For any $\theta_0\in\Theta$, $H_0\in(1/2, 1)$, and $\mathcal{H}\in(1/2,1)$, we have that
$\frac{1}{\sqrt{\epsilon}}\left(\hat\theta^{\mathcal{H}}_{MCE}(\{x_\tnk\}_{k=1}^n)-\theta_0\right)$ converges in distribution as $\varepsilon := (\epsilon, \eta) \to 0$ to the normal distribution $\mathcal{N}(0, M^\mathcal{H}(\theta_0, H_0; n))$, where the variance is given by
\begin{align*}
M^{\mathcal{H}}(\theta, H; n)
 &:= \left(\left( \oplus_{k=1}^n \nabla_\theta\Big|_{\theta_0} \bar X^\cdot_\tnk \right)^T \left( \XianzerocurlyH \right)^{-1} \left( \oplus_{k=1}^n \nabla_\theta\Big|_{\theta_0} \bar X^\cdot_\tnk \right)\right)^{-1} \times \nonumber\\
 & \hspace{4pc} \left( \oplus_{k=1}^n \nabla_\theta\Big|_{\theta_0} \bar X^\cdot_\tnk \right)^T \XianzerocurlyH \left( \Xianzero \right)^{-1} \XianzerocurlyH \left( \oplus_{k=1}^n \nabla_\theta\Big|_{\theta_0} \bar X^\cdot_\tnk \right) \hspace{1pc} \times \nonumber\\
 &\hspace{4pc}  \left(\left( \oplus_{k=1}^n \nabla_\theta\Big|_{\theta_0} \bar X^\cdot_\tnk \right)^T \left( \XianzerocurlyH \right)^{-1} \left( \oplus_{k=1}^n \nabla_\theta\Big|_{\theta_0} \bar X^\cdot_\tnk \right)\right)^{-1}.
\end{align*}
\end{theorem}

\noindent\textit{Proof.} Let us suppress the data $\{X^{\epsilon, \theta_0, H_0}_\tnk\}_{k=1}^n$ and the parameter $\mathcal{H}$ in the estimator. The argument is nearly identical to that of the proof of Theorem \ref{normalitytfe} and so we provide a sketch only.

\begin{align*}
-\frac{1}{\sqrt\epsilon}\nabla_\theta\Big|_{\theta_0} \tilde U &= 2 \left( \frac{1}{\sqrt\epsilon} \oplus_{k=1}^n \left( \xazero_\tnk - \bar X^{\theta_0}_\tnk \right) \right)^T \left( \XianzerocurlyH \right)^{-1} \left( \oplus_{k=1}^n \nabla_\theta\Big|_{\theta_0} \bar X^\cdot_\tnk \right) + \mathcal{R}_I ,
\end{align*}
where the first term converges in distribution to $2 \left( \xianzero \right)^T \left( \XianzerocurlyH \right)^{-1} \left( \oplus_{k=1}^n \nabla_\theta\Big|_{\theta_0} \bar X^\cdot_\tnk \right)$ and $\mathcal{R}_I$ converges in probability to $0$.

\begin{align*}
\nabla^2_\theta\Big|_{\theta^\dagger} \tilde U &= 2 \left( \oplus_{k=1}^n \nabla_\theta\Big|_{\theta^\dagger} \bar X^\cdot_\tnk \right)^T \left( \XianthetadaggermathcalH \right)^{-1} \left( \oplus_{k=1}^n \nabla_\theta\Big|_{\theta_0} \bar X^\cdot_\tnk \right) + \mathcal{R}_{II} ,
\end{align*}
where the first term converges in probability to $2 \left( \oplus_{k=1}^n \nabla_\theta\Big|_{\theta_0} \bar X^\cdot_\tnk \right)^T \left( \XianzerocurlyH \right)^{-1} \left( \oplus_{k=1}^n \nabla_\theta\Big|_{\theta_0} \bar X^\cdot_\tnk \right)$ and $\mathcal{R}_{II}$ converges in probability to $0$.

The claim follows as in the proof of Theorem \ref{normalitytfe}.

\qed

In the next theorem we compare the limiting variance of the TFE studied in Section \ref{S:TFE1} with that of the estimator studied in this section. As we shall see, the estimator studied in this section has smaller limiting variance, at least when the Hurst index is known. However, as we have mentioned, it is considerably more computationally challenging to implement than is the TFE.
\begin{theorem}(Comparing the Asymptotic Variances)\label{T:VarianceComparisons}
For any $\theta\in\Theta$, $H\in(1/2,1)$, and $n\in\mathbb{N}$, $M(\theta,H;n)-M^H(\theta,H;n)$ is positive semidefinite, or what is equivalent, $(M(\theta,H;n))^{-1}-(M^H(\theta,H;n))^{-1}$ is negative semidefinite.
\end{theorem}

\noindent\textit{Proof.} We will show that the difference of inverses is negative semidefinite. Let us begin by writing $A$ for $\nabla_\theta\Big|_{\theta_0} \bar X^\cdot_\tnk$ and $B$ for a symmetric square root of $\Xian$.

In this notation, for any vector $v$ in the domain,
\begin{align*}
\langle v, (M(\theta,H;n))^{-1} v \rangle^2 &= \langle v, A^T A (A^T B^2 A)^{-1} A^T A v \rangle^2 \nonumber \\
&= \langle B^{-1} A v, B A (A^T B^2 A)^{-1} A^T A v \rangle^2 \nonumber \\
&\leq \langle B^{-1} A v, B^{-1} A v \rangle \langle B A (A^T B^2 A)^{-1} A^T A v, B A (A^T B^2 A)^{-1} A^T A v \rangle \nonumber \\
& = \langle v, (M^H(\theta,H;n))^{-1} v \rangle \langle v, (M(\theta,H;n))^{-1} v \rangle. \nonumber
\end{align*}

The claim follows immediately.

\qed

\section{Numerical Examples}\label{S:NumericalExamples}

We now present data from numerical simulations to illustrate the theory. In Subsection \ref{SS:ConstantVolExample}, we consider a model in which $\sigma \equiv 1$, i.e., the diffusion coefficient of the slow component is constant, while in Subsection \ref{SS:VariableVolExample}, we consider a model with multiplicative noise in the slow component, i.e., $\sigma$ is nonconstant. We have constructed the models so as to exhibit the same behavior in the slow component, to two asymptotic orders, as $\varepsilon := (\epsilon,\eta) \to 0$. In addition, the diffusion coefficient $\sigma$ and fast component $Y$ in the second model are such that the averaged diffusion coefficient is precisely $\bar \sigma \equiv 1$, in agreement with the constant value of $\sigma$ in the first example. This design perhaps facilitates a comparative analysis of the two examples.

Before presenting the models and their statistical analysis, let us collect here some of the more high-level conclusions.
\begin{itemize}
\item{Estimation of $\theta$ is stable across all values of $n$ (even for small values). Moreover, as one would expect, data corresponding to smaller values of $\epsilon$ result in more accurate estimates of $\theta$.}
\item{When $\sigma$ is constant, both estimators for the Hurst parameter $H$, $\hat{H}^\epsilon_1$ and $\hat{H}^\epsilon_2$, work equally well as far as the point estimates are concerned, but as expected $\hat{H}^\epsilon_1$ has smaller variance than $\hat{H}^\epsilon_2$.}
\item{If $\sigma$ is variable, then $\hat{H}^\epsilon_2$ is in general more reliable than $\hat{H}^\epsilon_1$. A likely reason for this behavior lies in the fact that the construction of $\hat{H}^\epsilon_1$ uses the the averaged $\bar{\sigma}$ whereas the actual data are of course generated with the prelimit $\sigma$. Note that the estimator $\hat{H}^\epsilon_2$ does not require knowledge of $\bar{\sigma}$. In connection with this, it is also worth pointing out that both Theorems \ref{T:ConvergenceForHatH1}, for $\hat{H}^\epsilon_1$, and \ref{T:ConvergenceForHatH2}, for $\hat{H}_2$, make assumptions about the relationships among $\epsilon,\eta,n$. It turns out that in practice there is some tension between asking for $\eta$ to be small (so that the fast dynamics behave ergodically) and at the same time hoping to observe convergence of the estimators (one must ensure that the number of samples $n$ is commensurately large).}
\end{itemize}
Let us now proceed with our two examples.
\subsection{Constant-$\sigma$ Model}\label{SS:ConstantVolExample}
We begin by considering the constant-$\sigma$ system
\begin{equation}
\begin{cases}
d \xe_t = \theta_0 \xe_t \ye_t \ye_t dt + \sqrt\epsilon dW^{H_0}_t \label{Eq:ConstantSigmaModel} \\
d \ye_t = - \frac{1}{\eta} \ye_t + \frac{1}{\sqrt\eta} dB_t
\end{cases}
\end{equation}
for $t\in[0,T=1]$ with $(\xe_0, \ye_0) = (1, 0) \in \mathbb{R}^2$.

The limit $\bar X$ of the slow process $\xe$ in (\ref{Eq:ConstantSigmaModel}) is given by $\bar X_t = e^{\frac{t\theta_0}{2}}$ and, of course, $\bar\sigma = 1$.

We fix $\theta_0 = 1$ and $H_0 = 0.85$ and consider the estimators $\hat \theta_{TFE}$, $\hat{H}^\epsilon_1$, and $\hat{H}_2$ defined in equations (\ref{Eq:TFEDefinition}), (\ref{Eq:HatHEpsilonOneDefinition}), and (\ref{Eq:HatHTwoDefinition}) respectively. For each estimator, for $\epsilon=0.1$ and $\epsilon=0.01$, and for selected values of $\eta$ and $n$ (or $2n$ in the case of $\hat{H}_2$), we perform $10,000$ simulations. The slow and fast trajectories are simulated according to an Euler-Maruyama scheme with $10^6$ evenly-spaced discrete time steps.

Tables \ref{constant_sigma_model_epsilon_0.1_trajectory_fitting_estimates_means}-\ref{constant_sigma_model_epsilon_0.01_second_type_hurst_estimates_stds} present the empirical means and standard deviations. 

Let us start with the estimator for $\theta$.
\begin{table}[H]
\begin{center}
\begin{tabular}{|c||c|c|c|c|c|}
\hline
* & $n = 10^6$ & $n=10^5$ & $n = 10^4$ & $n = 10^3$ & $n = 10^2$ \\
\hline
\hline
$\eta = 0.01$ & 0.93566 & 0.9368 & 0.9223 & 0.939 & 0.9309 \\
\hline
$\eta = 0.001$ & 0.93834 & 0.93437 & 0.93605 & 0.94778 & 0.94202 \\
\hline
$\eta = 0.0001$ & 0.94638 & 0.94862 & 0.94215 & 0.94506 & 0.94895 \\
\hline
\end{tabular}
\caption{Means of $\hat \theta_{TFE}$ for constant-$\sigma$ model with $\epsilon = 0.1$}\label{constant_sigma_model_epsilon_0.1_trajectory_fitting_estimates_means}
\end{center}
\end{table}

\begin{table}[H]
\begin{center}
\begin{tabular}{|c||c|c|c|c|c|}
\hline
* & $n = 10^6$ & $n=10^5$ & $n = 10^4$ & $n = 10^3$ & $n = 10^2$ \\
\hline
\hline
$\eta = 0.01$ & 0.58072 & 0.58527 & 0.59206 & 0.58791 & 0.59172 \\
\hline
$\eta = 0.001$ & 0.57269 & 0.57432 & 0.56904 & 0.56154 & 0.56128 \\
\hline
$\eta = 0.0001$ & 0.56523 & 0.56427 & 0.56397 & 0.56596 & 0.56001 \\
\hline
\end{tabular}
\caption{Standard deviations of $\hat \theta_{TFE}$ for constant-$\sigma$ model with $\epsilon = 0.1$}\label{constant_sigma_model_epsilon_0.1_trajectory_fitting_estimates_stds}
\end{center}
\end{table}

\begin{table}[H]
\begin{center}
\begin{tabular}{|c||c|c|c|c|c|}
\hline
* & $n = 10^6$ & $n=10^5$ & $n = 10^4$ & $n = 10^3$ & $n = 10^2$ \\
\hline
\hline
$\eta = 0.01$ & 0.98719 & 0.98777 & 0.98919 & 0.98619 & 0.9852 \\
\hline
$\eta = 0.001$ & 0.99307 & 0.99424 & 0.99702 & 0.99042 & 0.9954 \\
\hline
$\eta = 0.0001$ & 1.00236 & 0.99737 & 1.00063 & 1.00136 & 0.99745 \\
\hline
\end{tabular}
\caption{Means of $\hat \theta_{TFE}$ for constant-$\sigma$ model with $\epsilon = 0.01$}\label{constant_sigma_model_epsilon_0.01_trajectory_fitting_estimates_means}
\end{center}
\end{table}

\begin{table}[H]
\begin{center}
\begin{tabular}{|c||c|c|c|c|c|}
\hline
* & $n = 10^6$ & $n=10^5$ & $n = 10^4$ & $n = 10^3$ & $n = 10^2$ \\
\hline
\hline
$\eta = 0.01$ & 0.22959 & 0.23078 & 0.23264 & 0.23139 & 0.2322 \\
\hline
$\eta = 0.001$ & 0.1795 & 0.17761 & 0.17749 & 0.18082 & 0.17798 \\
\hline
$\eta = 0.0001$ & 0.17353 & 0.17266 & 0.17173 & 0.17366 & 0.17349 \\
\hline
\end{tabular}
\caption{Standard deviations of $\hat \theta_{TFE}$ for constant-$\sigma$ model with $\epsilon = 0.01$}\label{constant_sigma_model_epsilon_0.01_trajectory_fitting_estimates_stds}
\end{center}
\end{table}

Comparing Tables \ref{constant_sigma_model_epsilon_0.1_trajectory_fitting_estimates_means} and \ref{constant_sigma_model_epsilon_0.01_trajectory_fitting_estimates_means} we see that when $\epsilon$ is smaller the estimates of $\theta$ are more accurate. Note also that the TFE is stable across all of the different values of $n$. In addition, we compute using the theoretical limiting standard deviation of Theorem \ref{normalitytfe} an approximate standard deviation of $0.54037$ for $\epsilon=0.1$ and $0.17088$ for $\epsilon=0.01$, both of which are very close to their empirical estimates presented in Tables \ref{constant_sigma_model_epsilon_0.1_trajectory_fitting_estimates_stds} and \ref{constant_sigma_model_epsilon_0.01_trajectory_fitting_estimates_stds} respectively.

Now we proceed with the two Hurst-index estimators.
\begin{table}[H]
\begin{center}
\begin{tabular}{|c||c|c|c|c|c|}
\hline
* & $n = 10^6$ & $n=10^5$ & $n = 10^4$ & $n = 10^3$ & $n = 10^2$ \\
\hline
\hline
$\eta = 0.01$ & 0.85 & 0.84997 & 0.84921 & 0.83503 & 0.7698 \\
\hline
$\eta = 0.001$ & 0.84998 & 0.84967 & 0.84356 & 0.81115 & 0.82201 \\
\hline
$\eta = 0.0001$ & 0.84978 & 0.84711 & 0.83132 & 0.83769 & 0.84763 \\
\hline
\end{tabular}
\caption{Means of $\hat{H}^{\epsilon}_{1}$ for constant-$\sigma$ model with $\epsilon = 0.1$}\label{constant_sigma_model_epsilon_0.1_first_type_hurst_estimates_means}
\end{center}
\end{table}
\begin{table}[H]
\begin{center}
\begin{tabular}{|c||c|c|c|c|c|}
\hline
* & $n = 10^6$ & $n=10^5$ & $n = 10^4$ & $n = 10^3$ & $n = 10^2$ \\
\hline
\hline
$\eta = 0.01$ & 5e-05 & 0.00017 & 0.00071 & 0.0059 & 0.02968 \\
\hline
$\eta = 0.001$ & 5e-05 & 0.0002 & 0.00204 & 0.01014 & 0.0145 \\
\hline
$\eta = 0.0001$ & 8e-05 & 0.00089 & 0.00486 & 0.00437 & 0.01058 \\
\hline
\end{tabular}
\caption{Standard deviations of $\hat{H}^{\epsilon}_{1}$ for constant-$\sigma$ model with $\epsilon = 0.1$}\label{constant_sigma_model_epsilon_0.1_first_type_hurst_estimates_stds}
\end{center}
\end{table}
\begin{table}[H]
\begin{center}
\begin{tabular}{|c||c|c|c|c|c|}
\hline
* & $n = 10^6$ & $n=10^5$ & $n = 10^4$ & $n = 10^3$ & $n = 10^2$ \\
\hline
\hline
$\eta = 0.01$ & 0.84998 & 0.84966 & 0.84283 & 0.77158 & 0.60576 \\
\hline
$\eta = 0.001$ & 0.84978 & 0.84684 & 0.80863 & 0.71165 & 0.71418 \\
\hline
$\eta = 0.0001$ & 0.84789 & 0.82844 & 0.76846 & 0.78014 & 0.8205 \\
\hline
\end{tabular}
\caption{Means of $\hat{H}^{\epsilon}_{1}$ for constant-$\sigma$ model with $\epsilon = 0.01$}\label{constant_sigma_model_epsilon_0.01_first_type_hurst_estimates_means}
\end{center}
\end{table}
\begin{table}[H]
\begin{center}
\begin{tabular}{|c||c|c|c|c|c|}
\hline
* & $n = 10^6$ & $n=10^5$ & $n = 10^4$ & $n = 10^3$ & $n = 10^2$ \\
\hline
\hline
$\eta = 0.01$ & 5e-05 & 0.00019 & 0.00177 & 0.01209 & 0.0391 \\
\hline
$\eta = 0.001$ & 5e-05 & 0.00041 & 0.00357 & 0.00974 & 0.02063 \\
\hline
$\eta = 0.0001$ & 0.00022 & 0.00171 & 0.00428 & 0.00575 & 0.01197 \\
\hline
\end{tabular}
\caption{Standard deviations of $\hat{H}^{\epsilon}_{1}$ for constant-$\sigma$ model with $\epsilon = 0.01$}\label{constant_sigma_model_epsilon_0.01_first_type_hurst_estimates_stds}
\end{center}
\end{table}

\begin{table}[H]
\begin{center}
\begin{tabular}{|c||c|c|c|c|c|}
\hline
* & $2n = 10^6$ & $2n=10^5$ & $2n = 10^4$ & $2n = 10^3$ & $2n = 10^2$ \\
\hline
\hline
$\eta = 0.01$ & 0.85005 & 0.851 & 0.86961 & 1.0264 & 0.83383 \\
\hline
$\eta = 0.001$ & 0.85046 & 0.85998 & 0.95922 & 0.85727 & 0.7368 \\
\hline
$\eta = 0.0001$ & 0.85441 & 0.91253 & 0.85565 & 0.79103 & 0.81603 \\
\hline
\end{tabular}
\caption{Means of $\hat{H}_{2}$ for constant-$\sigma$ model with $\epsilon = 0.1$}\label{constant_sigma_model_epsilon_0.1_second_type_hurst_estimates_means}
\end{center}
\end{table}

\begin{table}[H]
\begin{center}
\begin{tabular}{|c||c|c|c|c|c|}
\hline
* & $2n = 10^6$ & $2n=10^5$ & $2n = 10^4$ & $2n = 10^3$ & $2n = 10^2$ \\
\hline
\hline
$\eta = 0.01$ & 0.00145 & 0.00462 & 0.016 & 0.06725 & 0.17248 \\
\hline
$\eta = 0.001$ & 0.00145 & 0.00557 & 0.03151 & 0.05317 & 0.15745 \\
\hline
$\eta = 0.0001$ & 0.00197 & 0.01776 & 0.01581 & 0.04978 & 0.1502 \\
\hline
\end{tabular}
\caption{Standard deviations of $\hat{H}_{2}$ for constant-$\sigma$ model with $\epsilon = 0.1$}\label{constant_sigma_model_epsilon_0.1_second_type_hurst_estimates_stds}
\end{center}
\end{table}

\begin{table}[H]
\begin{center}
\begin{tabular}{|c||c|c|c|c|c|}
\hline
* & $2n = 10^6$ & $2n=10^5$ & $2n = 10^4$ & $2n = 10^3$ & $2n = 10^2$ \\
\hline
\hline
$\eta = 0.01$ & 0.85045 & 0.86004 & 0.99699 & 1.2951 & 0.82796 \\
\hline
$\eta = 0.001$ & 0.8544 & 0.93419 & 1.23565 & 0.8622 & 0.57189 \\
\hline
$\eta = 0.0001$ & 0.88917 & 1.15074 & 0.86317 & 0.63009 & 0.72459 \\
\hline
\end{tabular}
\caption{Means of $\hat{H}_{2}$ for constant-$\sigma$ model with $\epsilon = 0.01$}\label{constant_sigma_model_epsilon_0.01_second_type_hurst_estimates_means}
\end{center}
\end{table}

\begin{table}[H]
\begin{center}
\begin{tabular}{|c||c|c|c|c|c|}
\hline
* & $2n = 10^6$ & $2n=10^5$ & $2n = 10^4$ & $2n = 10^3$ & $2n = 10^2$ \\
\hline
\hline
$\eta = 0.01$ & 0.00146 & 0.00523 & 0.03152 & 0.05684 & 0.19151 \\
\hline
$\eta = 0.001$ & 0.00154 & 0.01011 & 0.02237 & 0.06417 & 0.17704 \\
\hline
$\eta = 0.0001$ & 0.00394 & 0.01501 & 0.01985 & 0.05434 & 0.15789 \\
\hline
\end{tabular}
\caption{Standard deviations of $\hat{H}_{2}$ for constant-$\sigma$ model with $\epsilon = 0.01$}\label{constant_sigma_model_epsilon_0.01_second_type_hurst_estimates_stds}
\end{center}
\end{table}

As the Tables above show, both estimators do reasonably well in this case, becoming more accurate as $n$ increases. The theoretical limiting standard deviations for the two Hurst-index estimators are given by Tables \ref{theoretical sd H1}-\ref{theoretical sd_H2} below and are reasonably close to the empirical ones, especially when $\epsilon=0.1$.
\begin{table}[H]
\begin{center}
\begin{tabular}{|c|c|c|c|c|}
\hline
 $n = 10^6$ & $n=10^5$ & $n = 10^4$ & $n = 10^3$ & $n = 10^2$ \\
\hline
\hline
 6e-05 & 0.00022 & 0.00085 & 0.00358 & 0.017 \\
\hline
\end{tabular}
\caption{Theoretical limiting standard deviations of $\hat H^\epsilon_1$  for all combinations of $\epsilon, \eta$.}
\label{theoretical sd H1}
\end{center}
\end{table}
\begin{table}[H]
\begin{center}
\begin{tabular}{|c|c|c|c|c|}
\hline
 $n = 10^6$ & $n=10^5$ & $n = 10^4$ & $n = 10^3$ & $n = 10^2$ \\
\hline
\hline
0.00145 & 0.00459 & 0.0145 & 0.04585 & 0.14499 \\
\hline
\end{tabular}
\caption{Theoretical limiting standard deviations of $\hat H_2$  for all combinations of $\epsilon,\eta$.}
\label{theoretical sd_H2}
\end{center}
\end{table}

\subsection{Variable-$\sigma$ Model}\label{SS:VariableVolExample}
We now consider the variable-$\sigma$ system
\begin{equation}
\begin{cases}
d \xe_t = \frac{1}{2} \theta_0 \xe_t dt + \sqrt\epsilon \frac{L}{2\pi} e^{\sin(\ye_t) + \cos(\ye_t)} dW^{H_0}_t \label{Eq:VariableSigmaModel} \\
d \ye_t = \frac{1}{2\eta} (\sin(\ye_t) - \cos(\ye_t)) dt + \frac{1}{\sqrt\eta} dB_t
\end{cases}
\end{equation}
for $t\in[0,T=1]$, where $L:=\int^{2\pi}_{0} e^{-(\sin(y)+\cos(y))} dy = \int^{2\pi}_{0} e^{\sin(y)+\cos(y)} dy$. For the purposes of averaging we regard the fast component as taking values in the circle $S$ obtained as a quotient of $\mathbb{R}$ upon identifying points whose distance from one another is an integral multiple of $2\pi$. Permitting a slight abuse of notation, let $(\xe_0, \ye_0) = (1, 0) \in \mathbb{R} \times S$.

The scaling is chosen so that once again the limit $\bar X$ of the slow process $\xe$ in (\ref{Eq:VariableSigmaModel}) is given by $\bar X_t = e^{\frac{t\theta_0}{2}}$ and $\bar\sigma = 1$, facilitating comparison of the experimental results between the two models.

We again fix $\theta_0 = 1$ and $H_0 = 0.85$ and consider the estimators $\hat \theta_{TFE}$ and $\hat{H}_2$ defined in equations (\ref{Eq:TFEDefinition}), (\ref{Eq:HatHEpsilonOneDefinition}), and (\ref{Eq:HatHTwoDefinition}) respectively. For each estimator, for $\epsilon=0.1$ and $\epsilon=0.01$, and for selected values of $\eta$ and $n$ (or $2n$ in the case of $\hat{H}_2$), we perform $10,000$ simulations. The slow and fast trajectories are simulated according to an Euler-Maruyama scheme with $10^6$ evenly-spaced discrete time steps.


Tables \ref{variable_sigma_model_epsilon_0.1_trajectory_fitting_estimates_means}-\ref{variable_sigma_model_epsilon_0.01_second_type_hurst_estimates_stds} present the empirical means and standard deviations. 
We start with the TFE $\hat \theta_{TFE}$ estimator.
\begin{table}[H]
\begin{center}
\begin{tabular}{|c||c|c|c|c|c|}
\hline
* & $n = 10^6$ & $n=10^5$ & $n = 10^4$ & $n = 10^3$ & $n = 10^2$ \\
\hline
\hline
$\eta = 0.01$ & 0.92498 & 0.9297 & 0.91818 & 0.92298 & 0.91503 \\
\hline
$\eta = 0.001$ & 0.94111 & 0.94252 & 0.93942 & 0.93276 & 0.93712 \\
\hline
$\eta = 0.0001$ & 0.9289 & 0.93416 & 0.94348 & 0.93748 & 0.94499 \\
\hline
\end{tabular}
\caption{Means of $\hat \theta_{TFE}$ for variable-$\sigma$ model with $\epsilon = 0.1$}\label{variable_sigma_model_epsilon_0.1_trajectory_fitting_estimates_means}
\end{center}
\end{table}

\begin{table}[H]
\begin{center}
\begin{tabular}{|c||c|c|c|c|c|}
\hline
* & $n = 10^6$ & $n=10^5$ & $n = 10^4$ & $n = 10^3$ & $n = 10^2$ \\
\hline
\hline
$\eta = 0.01$ & 0.63444 & 0.64771 & 0.65125 & 0.6456 & 0.65453 \\
\hline
$\eta = 0.001$ & 0.57226 & 0.579 & 0.58357 & 0.58011 & 0.5778 \\
\hline
$\eta = 0.0001$ & 0.56799 & 0.57611 & 0.56143 & 0.57273 & 0.56516 \\
\hline
\end{tabular}
\caption{Standard deviations of $\hat \theta_{TFE}$ for variable-$\sigma$ model with $\epsilon = 0.1$}\label{variable_sigma_model_epsilon_0.1_trajectory_fitting_estimates_stds}
\end{center}
\end{table}

\begin{table}[H]
\begin{center}
\begin{tabular}{|c||c|c|c|c|c|}
\hline
* & $n = 10^6$ & $n=10^5$ & $n = 10^4$ & $n = 10^3$ & $n = 10^2$ \\
\hline
\hline
$\eta = 0.01$ & 0.98835 & 0.99125 & 0.99326 & 0.99246 & 0.99009 \\
\hline
$\eta = 0.001$ & 0.99302 & 0.99112 & 0.99112 & 0.99488 & 0.99388 \\
\hline
$\eta = 0.0001$ & 0.99503 & 0.99587 & 0.99252 & 0.99287 & 0.99446 \\
\hline
\end{tabular}
\caption{Means of $\hat \theta_{TFE}$ for variable-$\sigma$ model with $\epsilon = 0.01$}\label{variable_sigma_model_epsilon_0.01_trajectory_fitting_estimates_means}
\end{center}
\end{table}

\begin{table}[H]
\begin{center}
\begin{tabular}{|c||c|c|c|c|c|}
\hline
* & $n = 10^6$ & $n=10^5$ & $n = 10^4$ & $n = 10^3$ & $n = 10^2$ \\
\hline
\hline
$\eta = 0.01$ & 0.20009 & 0.19998 & 0.19965 & 0.19787 & 0.1979 \\
\hline
$\eta = 0.001$ & 0.17634 & 0.17853 & 0.17669 & 0.17456 & 0.17539 \\
\hline
$\eta = 0.0001$ & 0.17297 & 0.17298 & 0.17149 & 0.1724 & 0.17242 \\
\hline
\end{tabular}
\caption{Standard deviations of $\hat \theta_{TFE}$ for variable-$\sigma$ model with $\epsilon = 0.01$}\label{variable_sigma_model_epsilon_0.01_trajectory_fitting_estimates_stds}
\end{center}
\end{table}

Our conclusions here are the same as those for the TFE in the constant-$\sigma$ example of Subsection \ref{SS:ConstantVolExample}. Note that because we have constructed the two examples so that certain limiting quantities coincide, the theoretical standard deviation values for the TFE $\hat \theta_{TFE}$ are the same as those given in \ref{SS:ConstantVolExample}.

Let us now proceed with the Hurst-index estimators.
\begin{table}[H]
\begin{center}
\begin{tabular}{|c||c|c|c|c|c|}
\hline
* & $n = 10^6$ & $n=10^5$ & $n = 10^4$ & $n = 10^3$ & $n = 10^2$ \\
\hline
\hline
$\eta = 0.01$ & 0.82242 & 0.81812 & 0.8117 & 0.80079 & 0.77864 \\
\hline
$\eta = 0.001$ & 0.82287 & 0.81827 & 0.81088 & 0.79324 & 0.79413 \\
\hline
$\eta = 0.0001$ & 0.82255 & 0.81711 & 0.8041 & 0.80577 & 0.83329 \\
\hline
\end{tabular}
\caption{Means of $\hat H^\epsilon_1$ for variable-$\sigma$ model with $\epsilon = 0.1$}\label{variable_sigma_model_epsilon_0.1_first_type_hurst_estimates_means}
\end{center}
\end{table}

\begin{table}[H]
\begin{center}
\begin{tabular}{|c||c|c|c|c|c|}
\hline
* & $n = 10^6$ & $n=10^5$ & $n = 10^4$ & $n = 10^3$ & $n = 10^2$ \\
\hline
\hline
$\eta = 0.01$ & 0.01364 & 0.01615 & 0.01941 & 0.02536 & 0.04578 \\
\hline
$\eta = 0.001$ & 0.00446 & 0.00527 & 0.0066 & 0.0114 & 0.0261 \\
\hline
$\eta = 0.0001$ & 0.00143 & 0.00172 & 0.00303 & 0.00693 & 0.01275 \\
\hline
\end{tabular}
\caption{Standard deviations of $\hat H^\epsilon_1$ for variable-$\sigma$ model with $\epsilon = 0.1$}\label{variable_sigma_model_epsilon_0.1_first_type_hurst_estimates_stds}
\end{center}
\end{table}

\begin{table}[H]
\begin{center}
\begin{tabular}{|c||c|c|c|c|c|}
\hline
* & $n = 10^6$ & $n=10^5$ & $n = 10^4$ & $n = 10^3$ & $n = 10^2$ \\
\hline
\hline
$\eta = 0.01$ & 0.82251 & 0.81802 & 0.81142 & 0.80042 & 0.77952 \\
\hline
$\eta = 0.001$ & 0.82286 & 0.81826 & 0.81072 & 0.79297 & 0.79494 \\
\hline
$\eta = 0.0001$ & 0.82254 & 0.81707 & 0.80411 & 0.80571 & 0.83334 \\
\hline
\end{tabular}
\caption{Means of $\hat H^\epsilon_1$ for variable-$\sigma$ model with $\epsilon = 0.01$}\label{variable_sigma_model_epsilon_0.01_first_type_hurst_estimates_means}
\end{center}
\end{table}

\begin{table}[H]
\begin{center}
\begin{tabular}{|c||c|c|c|c|c|}
\hline
* & $n = 10^6$ & $n=10^5$ & $n = 10^4$ & $n = 10^3$ & $n = 10^2$ \\
\hline
\hline
$\eta = 0.01$ & 0.01377 & 0.01605 & 0.01918 & 0.02555 & 0.04554 \\
\hline
$\eta = 0.001$ & 0.00447 & 0.00522 & 0.00657 & 0.01132 & 0.02596 \\
\hline
$\eta = 0.0001$ & 0.00142 & 0.00171 & 0.00309 & 0.00695 & 0.01272 \\
\hline
\end{tabular}
\caption{Standard deviations of $\hat H^\epsilon_1$ for variable-$\sigma$ model with $\epsilon = 0.01$}\label{variable_sigma_model_epsilon_0.01_first_type_hurst_estimates_stds}
\end{center}
\end{table}

\begin{table}[H]
\begin{center}
\begin{tabular}{|c||c|c|c|c|c|}
\hline
* & $2n = 10^6$ & $2n=10^5$ & $2n = 10^4$ & $2n = 10^3$ & $2n = 10^2$ \\
\hline
\hline
$\eta = 0.01$ & 0.85002 & 0.85028 & 0.85186 & 0.86408 & 0.81163 \\
\hline
$\eta = 0.001$ & 0.85015 & 0.85219 & 0.86907 & 0.86401 & 0.68595 \\
\hline
$\eta = 0.0001$ & 0.85134 & 0.87005 & 0.87218 & 0.69986 & 0.76573 \\
\hline
\end{tabular}
\caption{Means of $\hat{H}_{2}$ for variable-$\sigma$ model with $\epsilon = 0.1$}\label{variable_sigma_model_epsilon_0.1_second_type_hurst_estimates_means}
\end{center}
\end{table}

\begin{table}[H]
\begin{center}
\begin{tabular}{|c||c|c|c|c|c|}
\hline
* & $2n = 10^6$ & $2n=10^5$ & $2n = 10^4$ & $2n = 10^3$ & $2n = 10^2$ \\
\hline
\hline
$\eta = 0.01$ & 0.00418 & 0.01341 & 0.04129 & 0.12186 & 0.26647 \\
\hline
$\eta = 0.001$ & 0.00436 & 0.01365 & 0.04039 & 0.09693 & 0.16591 \\
\hline
$\eta = 0.0001$ & 0.00424 & 0.01309 & 0.03379 & 0.05615 & 0.15515 \\
\hline
\end{tabular}
\caption{Standard deviations of $\hat{H}_{2}$ for variable-$\sigma$ model with $\epsilon = 0.1$}\label{variable_sigma_model_epsilon_0.1_second_type_hurst_estimates_stds}
\end{center}
\end{table}

\begin{table}[H]
\begin{center}
\begin{tabular}{|c||c|c|c|c|c|}
\hline
* & $2n = 10^6$ & $2n=10^5$ & $2n = 10^4$ & $2n = 10^3$ & $2n = 10^2$ \\
\hline
\hline
$\eta = 0.01$ & 0.84999 & 0.85019 & 0.85115 & 0.86087 & 0.81679 \\
\hline
$\eta = 0.001$ & 0.85015 & 0.85238 & 0.86952 & 0.86384 & 0.68816 \\
\hline
$\eta = 0.0001$ & 0.8513 & 0.86992 & 0.87149 & 0.69979 & 0.76835 \\
\hline
\end{tabular}
\caption{Means of $\hat{H}_{2}$ for variable-$\sigma$ model with $\epsilon = 0.01$}\label{variable_sigma_model_epsilon_0.01_second_type_hurst_estimates_means}
\end{center}
\end{table}

\begin{table}[H]
\begin{center}
\begin{tabular}{|c||c|c|c|c|c|}
\hline
* & $2n = 10^6$ & $2n=10^5$ & $2n = 10^4$ & $2n = 10^3$ & $2n = 10^2$ \\
\hline
\hline
$\eta = 0.01$ & 0.00413 & 0.0131 & 0.04173 & 0.12224 & 0.26606 \\
\hline
$\eta = 0.001$ & 0.00428 & 0.01344 & 0.0406 & 0.09806 & 0.16667 \\
\hline
$\eta = 0.0001$ & 0.00426 & 0.01309 & 0.03325 & 0.05559 & 0.15577 \\
\hline
\end{tabular}
\caption{Standard deviations of $\hat{H}_{2}$ for variable-$\sigma$ model with $\epsilon = 0.01$}\label{variable_sigma_model_epsilon_0.01_second_type_hurst_estimates_stds}
\end{center}
\end{table}

We notice that $\hat{H}_{2}$ does a better job of correctly estimating $H$ than $\hat{H}_{1}$. This can be explained by noting that $\hat{H}_{1}$ depends on knowledge of the limiting coefficient $\bar{\sigma}$ whereas the actual data come from the prelimit model. It is therefore not surprising that the convergence here is slow. On the other hand $\hat{H}_{2}$ is based only on the variation in the sample, which may allow it to converge more quickly in such cases. Note that as for the TFE the theoretical standard deviation values are the same as those given in Subsection \ref{SS:ConstantVolExample}.

\appendix

\section{Auxiliary Results}\label{B:Appendix}

The following lemma bounds the moments of the maximum process of $|Y|$ with respect to $\eta$.

\vspace{1pc}
\begin{lemma}\label{ymaximal} Assume Conditions \ref{c:regularity} and \ref{c:recurrencebasic}. For any $ 0 \leq p < \infty $ and any $\zeta>0$,

\begin{align*}
E(\sup_{0\leq t\leq T}|\ye_t|^p)=o\left(\eta^{ - \zeta }\right)
\end{align*}
as $\eta\to0$.
\end{lemma}


\vspace{1pc}
\noindent\textit{Proof.} By Proposition 2 of \cite{pardoux2001poisson}, one has for the time-rescaled fast process, for any $ 0 \leq q < \infty $, the relation
\begin{align*}
E(\sup_{0\leq s\leq t}|\ye_{\eta s}|^q)=o(\sqrt t)
\end{align*}
as $t\to\infty$. This yields
\begin{align*}
E(\sup_{0\leq t\leq T}|\ye_t|^q)=o(\frac{1}{\sqrt\eta})
\end{align*}
as $\eta\to0$, whence the statement is immediate for $\zeta\geq1/2$. For $0<\zeta<1/2$, apply Jensen's inequality and substitute $q=p/2\zeta$:
\begin{align*}
E(\sup_{0\leq t\leq T}|\ye_t|^p)&\leq\left(E(\sup_{0\leq t\leq T}|\ye_t|^{p/2\zeta})\right)^{2\zeta}\\
&=o((\frac{1}{\sqrt\eta})^{2\zeta}).
\end{align*}

\qed

\section{Preliminaries}\label{A:Appendix}

\subsection{Fractional Brownian motion}\label{SS:fBm_preliminaries}

A fractional Brownian motion (fBm) is a centered Gaussian process $W^H =
\{ W^H_t \}_{ t \geq 0 } \subset L^2(\Omega)$, characterized by its
covariance function
\begin{equation*}
R_H(t,s) := E (W^H_t W^H_s) = \frac{1}{2} \left( s^{2H} + t^{2H} -
  \left\vert t-s \right\rvert^{2H} \right).
\end{equation*}
It is straightforward to verify that increments of fBm are stationary. The parameter $H \in (0,1)$ is usually referred to as the Hurst exponent, Hurst parameter, or Hurst index.

By Kolmogorov's continuity criterion, such a process admits a modification with continuous sample paths, and we always choose to work with such. In this case one may show in fact that almost every sample path is locally H\"older continuous of any order strictly less than $H$. It is this sense in which it is often said that the value of $H$ determines the regularity of the sample paths.

Note that when $H = \frac{1}{2}$, the covariance function is $R_{\frac{1}{2}}(t,s) = t \wedge s$. Thus, one sees that $W^{\frac{1}{2}}$ is a standard Brownian motion, and in particular that its disjoint increments are independent. In contrast to this, when $H \neq \frac{1}{2}$, nontrivial increments are not independent. In particular, when $H > \frac{1}{2}$, the process exhibits long-range dependence.

Note moreover that when $H \neq \frac{1}{2}$, the fractional Brownian motion is not a semimartingale, and the usual It\^o calculus therefore does not apply.

Another noteworthy property of fractional Brownian motion is that it
is self-similar in the sense that, for any constant $a >0$, the
processes $\left\{ W^H_t\right\}_{ t \geq 0 }$ and $\left\{ a^{-H}
  W^H_{at}\right\}_{ t \geq 0 }$ have the same distribution.

For more details about fractional Brownian motion, we refer the reader
to the monographs \cite{biagini_stochastic_2008, nourdin_selected_2012, nualart_malliavin_2006}.

The self-similarity and long-memory properties of the fractional
Brownian motion make it an interesting and suitable input noise in
many models in various fields such as analysis of financial time series,
hydrology, and telecommunications. However, in order to develop
interesting models based on fractional Brownian motion, one needs a
stochastic calculus with respect to the fBm, which will make use of
the stochastic calculus of variations, or Malliavin calculus, introduced
in the next subsection.

\subsection{Elements of Malliavin calculus}\label{SS:MalliavinCalculus_prelim}
\label{malliavinelements}
We outline here the main tools of Malliavin calculus needed in this
paper. For a complete treatment of this topic, we refer the reader to
\cite{nualart_malliavin_2006}.

Let $W^H = \left\{ W^H_t \right\}_{ t \geq 0 } \subset L^2(\Omega)$ be a fractional Brownian motion with Hurst index $H \in (\frac{1}{2},1)$ and let us fix a time interval $[0,T]$, where $T\in\mathbb{R}_+$.

The formula
\begin{equation*}
\left\langle \chi_{[0,s]},\chi_{[0,t]}\right\rangle_{\mathfrak{H}} := R_H(s,t)
\end{equation*}
induces an inner product on the set $\mathcal{E}$ of step functions on $[0, T]$. We denote by $\mathfrak{H}$ the Hilbert space obtained as the completion of the resulting inner product space.

It can be shown that the formula
\begin{equation}
\label{innerprodfbm}
\left\langle \varphi, \psi \right\rangle_{\mathfrak{H}} := \alpha_H \int_0^T \int_0^T \varphi(r)\psi(u)\left\vert r-u \right\rvert^{2H-2}du dr,
\end{equation}
with $\alpha_H := H(2H-1)$, extends the above inner product from $\mathcal{E}$ to the superset $L^2([0, T])$, and that it is equivalent to define $\mathfrak{H}$ as the completion of this extended inner product space (see e.g. \cite{decreusefond_stochastic_1999}).

Now, the map $\chi_{[0, t]} \mapsto W^H_t$ extends to a linear isometry of Hilbert spaces $\mathfrak{H} \to L^2(\Omega)$. We will denote this map also by $W^H$.

Recall that we are in the setting in which $H > \frac{1}{2}$. While one may interpret $\mathfrak{H}$ as a space of distributions, it has been shown in \cite{pipiras_integration_2000,pipiras_are_2001} that when $H > \frac{1}{2}$, the elements may not be ordinary functions but distributions of negative order. Adapting the inner product \eqref{innerprodfbm}, one can introduce the space $\vert \mathfrak{H} \vert$ of equivalence classes of measurable functions $\varphi$ on $[0, T]$ for which

\begin{equation*}
\left\lVert \varphi \right\rVert_{\vert \mathfrak{H} \vert}^2 := \alpha_H
\int_0^T \int_0^T \left\vert \varphi(r)
\right\rvert \left\vert \varphi(u)
\right\rvert \left\vert r-u
\right\rvert^{2H-2}du dr < \infty,
\end{equation*}
which is in fact a Banach space equipped with this square norm.

It can be shown that one has the following chain of continuous inclusions:
\begin{equation*}
L^2([0, T]) \subset L^{\frac{1}{H}}([0,T]) \subset \vert \mathfrak{H} \vert \subset \mathfrak{H}.
\end{equation*}

Let us now denote by $\mathcal{S}$ the set of smooth cylindrical random variables
of the form $F = f
\left( W^H(\varphi_1), \cdots , W^H(\varphi_n) \right)$, where $n \geq 1$, $\{\varphi_i\}^n_{i=1} \subset
\mathfrak{H}$, and $f \in C_b^{\infty} \left(
  \mathbb{R}^n \right)$ ($f$ and all of its partial derivatives of all orders are bounded functions).

The Malliavin derivative of such a smooth cylindrical random variable
$F$ is defined as the $\mathfrak{H}$-valued random variable given by
\begin{equation*}
DF := \sum_{i=1}^n \frac{\partial f}{\partial x_i} \left( W^H(\varphi_1),
  \cdots, W^H(\varphi_n) \right)\varphi_i.
\end{equation*}
The derivative operator $D$ is a closable operator from $L^2(\Omega)$
into $L^2(\Omega ; \mathfrak{H})$, and we continue to denote by $D$
the closure of the derivative operator, the domain of which we denote by $\mathbb{D}^{1,2}$, and which is a Hilbert space in the Sobolev-type norm
\begin{equation*}
\left\lVert F \right\rVert_{1,2}^2 := E (F^2) + E \left( \left\lVert DF \right\rVert_{\mathfrak{H}}^2 \right).
\end{equation*}

Similarly one obtains a derivative operator $D:\mathbb{D}^{1, 2}(\mathfrak{H}) \to L^2(\Omega; \mathfrak{H} \otimes \mathfrak{H})$ as the closure of $D:L^2(\Omega; \mathfrak{H}) \to L^2(\Omega; \mathfrak{H} \otimes \mathfrak{H})$, and so on.

Note that more generally with $p > 1$ one can analogously obtain $\mathbb{D}^{1, p}$ as Banach spaces of Sobolev type by working with $L^p(\Omega)$.
\\~\\
We can now introduce the divergence operator
$\delta$ as the adjoint of the derivative operator $D$. By definition, an
$\mathfrak{H}$-valued random variable $u \in L^2(\Omega ;
\mathfrak{H})$ is in the domain of $\delta$, which we denote by
$\operatorname{dom}\delta$, if there is a constant $c_u$ for which, for all $F \in \mathbb{D}^{1, 2}$,
\begin{equation*}
\left\vert E \left( \left\langle DF ,u \right\rangle_{\mathfrak{H}} \right)
\right\rvert \leq c_u \left\vert F \right\rvert_{L^2(\Omega)}.
\end{equation*}
For such an element $u$, $\delta(u)$ is defined
by duality as the unique element of $L^2(\Omega)$ such that, for each $F \in \mathbb{D}^{1, 2}$,
\begin{equation*}
E \left( F \delta(u) \right) = E \left( \left\langle DF, u \right\rangle_{\mathfrak{H}} \right).
\end{equation*}
It can be shown that
$\mathbb{D}^{1,2}(\mathfrak{H}) \subset \operatorname{dom}\delta$, and
that for any $u \in \mathbb{D}^{1,2}(\mathfrak{H})$,
\begin{equation*}
E \left( \delta(u)^2 \right) = E \left(\left\lVert u
  \right\rVert_{\mathfrak{H}}^2  \right) + E \left( \left\langle Du,
    (Du)^{*} \right\rangle_{\mathfrak{H}\otimes \mathfrak{H}} \right),
\end{equation*}
where $(Du)^{*}$ is the adjoint of $Du$ in the Hilbert space $\mathfrak{H}\otimes \mathfrak{H}$.\\

\subsection{Multiple Wiener integrals of deterministic functions with respect to fractional Brownian motion}

\subsection{Stochastic integration with respect to fractional Brownian
  motion}\label{Eq:StochasticIntergation_prelim}
~\\
In this subsection we state useful properties of multiple Wiener integrals of elements of $\mathfrak{H}$ with respect the fractional Brownian motion and introduce two main methods used to define stochastic integrals with respect to the fractional Brownian motion. These and other available approaches are collected and discussed in detail in the monograph \cite{biagini_stochastic_2008}.

The first method, introduced in \cite{decreusefond_stochastic_1999}, is based on the stochastic calculus of variations, or Malliavin calculus. Owing to the central role played by the divergence operator introduced in Subsection \ref{malliavinelements}, stochastic integrals of this type are commonly referred to as divergence integrals.

The second approach uses the fact that the H\"older regularity of the paths of fBm with $H > \frac{1}{2}$ is sufficient to allow integration in the sense of Z\"ahle \cite{zahle_integration_1998} or \cite{russo_forward_1993} (see also the classic paper \cite{young_inequality_1936}). Stochastic integrals of this type are often called pathwise integrals.
\begin{remark}
The divergence integral can be formulated for fractional Brownian motion with any $H \in (0, 1)$ whereas the pathwise integral exists only for $H > \frac{1}{2}$. One reason that we restrict attention to the case $H > \frac{1}{2}$ in this work is so that we may make use of known results for both.
\end{remark}

\subsubsection{Multiple Wiener integrals}\label{SS:MultipleWienerIntergation}
For $q \in \mathbb{N}$ and $f \in \frak{H}^{\otimes q}$, we denote by
$I_q(f)$ the multiple Wiener integral of order $q$ with respect to the
fractional Brownian motion $W^H$ introduced in the preceding subsections
(when several different fractional Brownian motions are being used, a
superscript will be added to the multiple integral notation to avoid
ambiguities as to which fractional Brownian motion is the
integrator). We refer the reader to \cite[Chapter 1]{nualart_malliavin_2006} for a definition and construction of these objects. We limit ourselves here to recalling the main properties of multiple Wiener integrals needed in this paper, the first of which is
that multiple Wiener integrals of different orders
are orthogonal in $L^2(\Omega)$. The multiple Wiener integrals form an algebra with the
following product rule: for $q,p \in \mathbb{N}$, $f \in
\frak{H}^{\otimes q}$ and $g \in \frak{H}^{\otimes p}$,
\begin{equation}
  \label{multwienerintproductrule}
  I_q(f)I_p(g) = \sum_{r= 0}^{q \wedge
    p}r!\binom{q}{r}\binom{p}{r}I_{p+q - 2r}\left( f \otimes_r g \right),
\end{equation}
where $f \otimes_0 g := f \otimes g$ and for each $1 \leq r \leq q\wedge
p$, the contraction $f \otimes_r g$ is given by
\begin{equation*}
f \otimes_r g := \sum_{i_1, \ldots,i_r =1}^{\infty}\left\langle f,
  e_{i_1}\otimes \ldots \otimes e_{i_r}
\right\rangle_{\frak{H}^{\otimes r}} \otimes \left\langle g,
  e_{i_1}\otimes \ldots \otimes e_{i_r}
\right\rangle_{\frak{H}^{\otimes r}}
\end{equation*}
where $\{ e_k \}^\infty_{k = 1}$ is any complete orthonormal
system in $\frak{H}$.

\subsubsection{Divergence integration}\label{SS:DivergenceIntergation}

The definition of the divergence operator as the adjoint of the Malliavin derivative operator suggests interpretation as an integral. Indeed, in the standard Brownian motion case ($H = \frac{1}{2}$), the divergence of an adapted, It\^o-integrable process coincides with its familiar It\^o integral. In general one defines, for $u \in \operatorname{dom} \delta$ and $0 \leq t \leq T$,
\begin{equation*}
\int_0^t u_s \delta W^H_s := \delta(u
\chi_{[0,t]}),
\end{equation*}
which we call the divergence integral of $u$. Note that the divergence integral is always centered in the sense that its expected value is zero.
\\~\\
We shall make use of a maximal inequality for the divergence integral, which we now state. The interested reader is referred to \cite{alosnualart} for more details.

Denote by $\mathbb{L}_H^{1,p}$ the set of elements
$u \in \mathbb{D}^{1,p}(\mathfrak{H})$ for which
\begin{equation*}
E \left(
  \left\vert u \right\rvert_{L^{\frac{1}{H}}([0,T])}^p + \left\vert
    Du \right\rvert_{L^{\frac{1}{H}}([0,T]^2)}^p \right) < \infty.
\end{equation*}
There is a constant $C$ depending only on $H$ and $T$ such that for any $p$ with $pH > 1$ and any $u \in \mathbb{L}_H^{1, p}$,
\begin{equation*}
E \left( \sup_{0 \leq t \leq T} \left\vert \int_0^t u_s \delta W^H_s
\right\rvert^p \right) \leq C \left[ \int_0^T \left\vert E \left( u_s \right)\right\rvert^p ds + \int_0^T E
\left( \int_0^T \left\vert D_s u_r
\right\rvert^{\frac{1}{H}}ds \right)^{pH}dr \right].
\end{equation*}
Here, $D u_r$ is being interpreted as a stochastic process and the subscript $s$ in the notation $D_s u_r$ refers to its parameter. Note that if we denote by $\lambda$ the Lebesgue measure on $[0, T]$, by $P$ the probability measure on $\Omega$, and by $\omega \in \Omega$ the random state, then $D_s u_r$ is defined for $\lambda \times P$-almost-every pair $(s, \omega)$.

\subsubsection{Pathwise integration}\label{SS:PathwiseIntergation}

We present a version of pathwise integration that appears by the name of symmetric stochastic integration in \cite{russo_forward_1993}.

Let $u = \left\{ u_t \right\}_{0 \leq t \leq T}$ be a stochastic
process in $\mathbb{D}^{1,2}(\mathfrak{H})$. If one has that
\begin{equation*}
E \left(\left\lVert u \right\rVert_{\vert \mathfrak{H} \vert}^2 +
  \left\lVert Du \right\rVert_{\vert \mathfrak{H} \vert \otimes \vert
    \mathfrak{H} \vert}^2  \right) < \infty
\end{equation*}
and
\begin{equation*}
\int_0^T \int_0^T \left\vert D_su_t \right\rvert \left\vert t-s
\right\rvert^{2H-2} ds dt < \infty \operatorname{~  a.s.},
\end{equation*}
then the symmetric integral
\begin{equation*}
\int_0^T u_t dW^H_t
\end{equation*}
defined as the limit in probability as $\varepsilon$ tends to zero of
\begin{equation*}
\frac{1}{2\varepsilon} \int_0^T u_s \left( W^H_{(s+\varepsilon) \wedge
  T} - W^H_{(s-\varepsilon) \vee 0}\right)ds
\end{equation*}
exists and for each $t \in [0,T]$,
\begin{equation}
  \label{relationpathwise-divergence}
\int_0^t u_s dW^H_s =  \int_0^t u_s \delta W^H_s + \alpha_H \int_0^t \int_0^T D_r u_s \left\vert s-r
\right\rvert^{2H-2} dr ds.
\end{equation}
Thus one sees how the pathwise and divergence integrals are related to one another. Note in particular that whereas the divergence integral is centered, the pathwise integral generally speaking is not. In the setting of the model in this paper, however, the two integrals coincide.
\begin{remark}
Note that whenever one has $Du = 0$, as is the case for instance when the integrand $u$ and the fractional Brownian motion $W^H$ are independent stochastic processes, the relation \eqref{relationpathwise-divergence} says
\begin{equation*}
\int_0^t u_s dW^H_s =  \int_0^t u_s \delta W^H_s,
\end{equation*}
which is to say that the two approaches lead to the same integral and in particular that both are centered.
\end{remark}

\end{document}